\documentclass[reqno, twoside, a4paper, 11pt]{amsart} 
\makeatletter

\@addtoreset{equation}{section}
\makeatother
\usepackage{geometry}
\usepackage[dvipdfmx]{graphicx}
\usepackage[dvipdfm]{hyperref}
\usepackage{amsmath}
\usepackage{amssymb}
\usepackage{latexsym}
\usepackage{amsthm}
\newtheorem{Thm}{Theorem}[section]

\newtheorem{Lem}[Thm]{Lemma}
\newtheorem{Prop}[Thm]{Proposition}

\newtheorem{Cor}[Thm]{Corollary}

\theoremstyle{definition}
\newtheorem{Def}[Thm]{Definition}

\newtheorem{Exa}[Thm]{Example}

\newcommand{\R}{\mathbb{R}} 

\begin{document}

\title{A new generalization of the Takagi function}
\author{Kazuki Okamura} 
\address{Research Institute for Mathematical Sciences, Kyoto University, Kyoto 606-8502, JAPAN}  
\email{{\tt kazukio@kurims.kyoto-u.ac.jp}} 


\begin{abstract} 
We consider a one-parameter family of 
functions $\{F(t,x)\}_{t}$ on $[0,1]$ 
and partial derivatives $\partial_{t}^{k} F(t, x)$ 
with respect to the parameter $t$.   
Each function of the class is defined by a certain pair of two square matrices of order two. 
The class includes the Lebesgue singular functions and other singular functions.    
Our approach to the Takagi function is similar to Hata and Yamaguti. 
The class of partial derivatives $\partial_{t}^{k} F(t, x)$
includes the original Takagi function and some generalizations.   
We consider real-analytic properties of $\partial_{t}^{k} F(t, x)$ as a function of $x$,  
specifically, 
differentiability, the Hausdorff dimension of the graph, 
the asymptotic around dyadic rationals, 
variation, 
a question of local monotonicity and 
a modulus  of continuity.   
Our results are extensions of some results for the original Takagi function and some generalizations.  
\end{abstract} 

\maketitle

\section{Introduction}

The Takagi function \cite{T}, which is denoted by $T$ throughout the paper,  is an example of continuous nowhere differentiable functions 
and has been considered from various points of view. 
Since $T$ is a fractal function, 
it is interesting to investigate real-analytic properties of $T$. 
For example, differentiability, the Hausdorff dimension of the graph, the asymptotic around dyadic rationals  and 
a modulus  of continuity  of $T$ have been considered.

Hata and Yamaguti \cite{HY} showed the following relation between the Takagi function $T(x)$ and the Lebesgue singular\footnote{In this paper a singular function is a continuous increasing function on $[0,1]$ whose derivatives are zero Lebesgue-a.e.}  function $L_{a}(x)$ with singularity parameter $a$ : 
\begin{equation}\label{HY-fund}
\frac{\partial}{\partial a}\bigg|_{a = 1/2} L_{a}(x) = T(x). 
\end{equation}
Now give a precise definition of $L_a$. 
Let $\mu_{a}$ be the probability measure on $\{0,1\}$ with $\mu_{a}(\{0\}) = a$ and $\mu_{a}^{\otimes\mathbb{N}}$ be the product measure of $\mu_{a}$ on $\{0,1\}^{\mathbb{N}}$.
Let $\varphi : \{0,1\}^{\mathbb{N}} \to [0,1]$  be a function defined by $\varphi((x_{n})_{n}) = \sum_{n=1}^{\infty} x_{n}/2^{n}$.
Let $L_{a}$ be the distribution function of the image measure of $\mu_{a}^{\otimes\mathbb{N}}$ by  $\varphi$. 
$L_a$ is identical with $\Phi_{2,1/a}$ in Paradis, Viader and  Bibiloni \cite{PVB}.

Recently, 
de Amo, D\'iaz Carrillo and  Fern\'andez-S\'anchez \cite{dA1} considered $\partial_{a}^n L_{a}(x)$ at $a \ne 1/2$. 
(Here and henceforth $\partial_{z}^{n}$ denotes the $n$-th partial derivatives 
with respect to the variable $z$. If $n=1$ write simply $\partial_{z}$.)  
They showed for any $a \ne 1/2$ and for $n \ge 1$,   
$\partial_{a}^{n} L_{a}(x)$ has zero derivative at almost every $x$. 
They claimed if $n$ is odd,  $\partial_{a}^{n} L_{a}$ is of monotonic type on no open interval (MTNI\footnote{We follow Brown, Darji and Larsen \cite{BDL} for this terminology.}). 
That is, on any open interval $J$ in $[0,1]$,
\[ -\infty = \inf_{x, y \in J, x \ne y} \frac{\partial_{a}^{n} L_{a}(x) - \partial_{a}^{n} L_{a}(y)}{x-y} < \sup _{x, y \in J, x \ne y} \frac{\partial_{a}^{n} L_{a}(x) - \partial_{a}^{n} L_{a}(y)}{x-y}= +\infty. \]

In this paper  
we consider a further generalization of $T$ by replacing $L_a$ in (\ref{HY-fund}) with more general functions and parametrizations. 
The author's paper  \cite{O1} considers a probability measure $\mu_{A_0, A_1}$ on $[0,1]$ defined by a certain pair of two $2 \times 2$ real matrices $(A_0, A_1)$. 
$\mu_{A_0, A_1}$ is singular or absolutely continuous with respect to the Lebesgue measure. 
The class of probability measures in \cite{O1} contains not only the Bernoulli measures but also many {\it non-product} measures\footnote{We identify $[0,1)$ with the Cantor space $\{0,1\}^{\mathbb{N}}$ in the natural way. We consider non-atomic measures on $[0,1]$ only  and we do not need to distinguish $[0,1]$ from $[0,1)$.}.   
Parametrize $(A_0, A_1)$ by a parameter $t$ around $0$.    
Assume each component of $A_0(t)$ and $A_1(t)$ is smooth\footnote{In this paper a smooth function is a function differentiable infinitely many times.} with respect to $t$ and $(A_0(0), A_1(0)) = (A_0, A_1)$.   
Denote the distribution function of $\mu_t$ by $F(t, \cdot)$. 
That is, $F(t, x) = \mu_{A_0(t), A_1(t)}([0,x]), x \in [0,1]$.   

The main subject of this paper is investigating real analytic properties for the $k$-th partial derivative $f_k(x) := \partial^{k}_{t}F(0,x)$. 
Our framework gives a generalization of $T$.  
$F(t, x) = L_{a+t}(x)$ for a specific choice of $(A_0(t), A_1(t))$. 
Thus our framework contains the one of \cite{dA1}. 
Our generalization is different from the ones by \cite{HY} and K\^ono \cite{Ko}. 
The graphs of these curves can be quite different, from Takagi's classical, $T$,  
to very {\it asymmetrical} ones as shown in figures \ref{fig1} and \ref{fig2} below.    
In Section 2 we will give the framework and show $f_k$ is well-defined and continuous on $[0,1]$ for each $k \ge 1$.

In Section 3 we will show the Hausdorff dimension of the graph of $f_k$ is $1$.  
This extends Allaart and Kawamura \cite[Corollary 4.2]{AK1} and is applicable to the framework in \cite[Section 5]{dA1}. 
Our proof is different from Mauldin and Williams \cite{MW} and \cite{AK1} and seems simpler than them because we do not need to investigate strength of continuity of $f_k$.     
In Section 4 we will show the derivative of $f_k$ is $0$ almost everywhere.    
This extends \cite[Theorems 12 and 13]{dA1}. 
We will examine the asymptotic of $f_k$ around dyadic rationals in Section 5.  
The asymptotic of $f_k$ around {\it dyadic rationals} and around {\it Lebesgue-a.e. points} can be similar on the one hand but can be considerably different on the other hand.   
As shown in Figure \ref{fig1} there is a fractal function whose derivatives are zero at all dyadic rationals.  
To our knowledge such a fractal function is unusual.    

If we consider the case $k=1$ and the ``linear" case, each of which contains the original Takagi function $T$, 
we have more sophisticated results. 
In Theorem \ref{N-Conv-V} we will consider differentiability and variation of $f_{k}$.  
\cite[Theorem 14]{dA1} 
states if we consider the ``linear" case and $k$ is odd, $f_k$ is MTNI.  
Theorem \ref{MTNI} will extend \cite[Theorem 14]{dA1} to \textit{all} $k \ge 1$.  
If $\mu_0$ is singular, the asymptotic of $f_k$ around {\it $\mu_0$-a.s. points} and around {\it Lebesgue-a.e. points} can be considerably different.   
In Section 7 we will consider a modulus  of continuity of $f_1$. 
Theorem \ref{E-AK-JMAA} will extend Allaart and Kawamura \cite[Theorem 5.4]{AK-JMAA}, 
which gives a necessary and sufficient condition for the existence of 
\[ \lim_{h \to 0} \frac{T(x+h) - T(x)}{h\log_2 (1/|h|)} \ \ \text{ at non-dyadic $x$. } \] 
Theorem \ref{LMC} will investigate a modulus of continuity of $f_1$ at {\it $\mu_0$-a.s. points}.  
It is similar to \cite{Ko}. 
We have the original Takagi function case\footnote{\cite{Ko} considers this in a general setting different from ours.}  of \cite{Ko} by our approach.   
Our proofs are different from \cite{AK-JMAA} and  \cite{Ko}.  
We do not use \cite[Lemma 3]{Ko} which plays an important role in \cite{AK-JMAA} and  \cite{Ko}.


\section{Framework}  

\subsection{Definition of $\mu_{A_{0}, A_{1}}$} 

Let $A_{i} = \begin{pmatrix} a_i & b_i \\ c_i & d_i \end{pmatrix}$, $i = 0,1$, 
be two real $2 \times 2$ matrices such that the following hold  :  \\
(i) $0 = b_0 < \dfrac{a_0 + b_0}{c_0 + d_0} = \dfrac{b_1}{d_1} < \dfrac{a_1 + b_1}{c_1 + d_1}  = 1$. \\
(ii) $a_i d_i - b_i d_i > 0$, $i = 0,1$. \\
(iii) $(a_i d_i - b_i d_i )^{1/2} < \min\{c_i, c_i + d_i\}$, $i = 0,1$. 

Consider a functional equation for $f : [0,1] \to \R$ : 
\begin{equation}\label{Def} 
	f(x) = \begin{cases}
						\Phi\left(A_{0}, f(2x)\right) & 0 \le x \le 1/2 \\
						\Phi\left(A_{1},  f(2x-1)\right) & 1/2  \le x \le 1, 
				          \end{cases}
				          \ \ \text{where } \Phi(A, z) := \dfrac{az+b}{cz+d} \text{ for } A = \begin{pmatrix} a & b\\ c & d \end{pmatrix}.  
\end{equation}
Conditions (i) - (iii) assure the existence of 
a unique continuous solution for (\ref{Def}). 
(\ref{Def}) is a special case of de Rham's functional equations \cite{dR}. 
Let $\mu_{A_{0}, A_{1}}$ be the measure such that the unique continuous solution $f$ of (\ref{Def}) is the distribution function of $\mu_{A_{0}, A_{1}}$.  
By conditions (i) - (iii) 
we can represent all components of $A_0, A_1$ by $b_1, c_0$ and $c_1$.  
We can assume  $d_{0} = d_{1} = 1$. 
Conditions (i) - (iii) imply $a_0 = b_{1}(c_{0}+1)$, $b_0 = 0$, $a_1 = 1- b_{1} + c_{1}$ and  
\begin{equation}\label{Cond} 
b_{1} \in (0,1), 
c_{0} \in \left(b_{1}-1, \frac{1}{b_{1}}-1\right), 
c_{1} \in \left(-b_{1}, \frac{b_{1}}{1-b_{1}}\right).   
\end{equation} 

If $b_{1} = a$ and $c_{0} = c_{1} = 0$  the Lebesgue singular function $L_{a}$ is the distribution function of $\mu_{A_{0}, A_{1}}$.   
$c_{0} = c_{1} = 0$ if and only if both $\Phi(A_{0}; \cdot)$ and $\Phi(A_{1}; \cdot)$ are  linear functions. 
By \cite[Theorem 1.2]{O1},  
$\mu_{A_{0}, A_{1}}$ is absolutely continuous if $c_0 = (2b_1)^{-1} - 1$ and $c_1 = 1-2b_1$, and singular otherwise. 
Let 
\[ \alpha := \min\left\{0, \frac{c_{0}}{1-b_{1}(c_{0}+1)}, \frac{c_{1}}{b_{1}} \right\}   
\text{ and } \beta := \max\left\{0, \frac{c_{0}}{1-b_{1}(c_{0}+1)}, \frac{c_{1}}{b_{1}} \right\}. \] 

$\alpha = \beta = 0$ if and only if $c_0 = c_1 = 0$. 
Roughly speaking $\alpha$ and $\beta$ measure how $\mu_{A_0, A_1}$ is ``far" from the Bernoulli measures.

Now define a ``dual" $(\widetilde A_{0}, \widetilde A_{1})$ associated with $(A_0, A_1)$ in order to shorten some proofs. 
\begin{Def}[Dual matrices]\label{Dual-Def}  
Let 
\begin{equation}\label{dual-Def}
(\widetilde b_1, \widetilde c_0, \widetilde c_1) := \left(1-b_1, -\frac{c_1}{1+c_1}, -\frac{c_0}{1+c_0} \right), 
\end{equation}   
Define $\widetilde A_{i}, i = 0,1$, $\widetilde\alpha$ and $\widetilde\beta$ 
by substituting $(\widetilde b_1, \widetilde c_0, \widetilde c_1)$ for $(b_1, c_0, c_1)$
in the definition of $A_{i}$, $\alpha$ and $\beta$.    
(\ref{Cond}) holds for  $(\widetilde A_0, \widetilde A_1)$ if and only if it holds for $(A_0, A_1)$.   
We have
\begin{equation}\label{Dual}
\mu_{\widetilde A_{0}, \widetilde A_1}([0,x]) = \mu_{A_{0},  A_1}([1-x, 1]), \, x \in [0,1]. 
\end{equation} 
\begin{equation}\label{Double}
\widetilde{\widetilde{A_i}} = A_i, \, i=0,1.   
\end{equation} 
\end{Def}

\subsection{Parametrization}

(1) In addition to (\ref{Cond}) 
we assume either the Lipschitz constant of $\Phi\left(^{t} \! A_{1}; y\right)$ on $y \in [\alpha, \beta]$ or   
the Lipschitz constant of $\Phi\left(^{t} \! \widetilde A_{1}; y\right)$ on $y \in \left[ \widetilde\alpha, \widetilde\beta \right]$ is strictly less than $1$. 
That is 
\begin{equation} \label{Ass}
(1+c_{1})\left(1-b_{1}(1+c_{0})\right)^{2} < 1-b_{1} \text{ or } (b_1 + c_1)^2 < b_1(1+c_0)(1+c_1).    
\end{equation} 

Assume this condition 
by a difficulty arising in computation in Lemma \ref{Bdd} below.  
However if $c_{0} = c_{1} = 0$,    
(\ref{Ass}) holds.   
The Lipschitz constant of $\Phi\left(^{t} \! A_{0}; y\right)$ on $y \in [\alpha, \beta]$ and  
the Lipschitz constant of $\Phi\left(^{t} \! \widetilde A_{0}; y\right)$ on $y \in \left[\widetilde\alpha, \widetilde\beta\right]$ are strictly less than $1$.

(2) Conditions (\ref{Cond}) and (\ref{Ass}) define an open set $E$ in $\R^3$ in which  we will consider different curves.  
\begin{align*} 
E := &\left\{(x, y, z) \in \R^{3} \ \bigg|  \ 0 < x < 1,  x -1 < y < \frac{1-x}{x}, -x < z < \frac{x}{1-x} \right\} \\
& \ \ \cap \bigl\{ (x, y, z) \big|  (1+z)(1-x(1+y))^{2} < 1-x \text{ or } (x+z)^2 < x(1+y)(1+z) \bigr\}. 
\end{align*}   

(3) Fix a point $(b_0, c_0, c_1) \in E$. 
We consider a smooth curve $(b_{1}(t), c_{0}(t), c_{1}(t))$ in $E$ on an open interval containing $0$ 
such that $(b_{1}(0), c_{0}(0), c_{1}(0)) = (b_0, c_0, c_1)$.

Define $A_0(t), A_1(t), \alpha(t), \beta(t)$ 
by substituting $(b_{1}(t), c_{0}(t), c_{1}(t))$ for $(b_1, c_0, c_1)$ 
in the definition of $A_0, A_1, \alpha, \beta$.       
Let 
\begin{equation*}
\mu_{t} := \mu_{A_{0}(t), A_{1}(t)} \text{ and } F(t, x) := \mu_{t}([0,x]),  \  \  x \in [0,1].
\end{equation*}    

This class of smooth curves includes the frameworks of \cite{AK1}, \cite{AK-JMAA}, \cite{dA1} and \cite{HY}.   
We have 
\[ \{(x,0,0) : 0 < x < 1\} \subset E. \]
If $(b_1(t), c_{0}(t), c_{1}(t)) = (t+a, 0, 0)$,  $F(0, x) = L_{a}(x)$.

\subsection{Notation and lemma} 

Let $X_{i}(x) := z_{i}$ if $x = \sum_{n \ge 1} 2^{-n} z_{n}$ is the dyadic expansion of $x$\footnote{As usual we assume the number of $n$ with $z_n = 1$ is finite.}.  

\begin{Def}\label{Def-many}
\textup{(i)}  \ 
\begin{equation}\label{Def-G}
G_{j}(t, y) := \Phi\left(^{t}\!A_{j}(t); y\right),   \ y \in [\alpha(t), \beta(t)],  j = 0,1.  
\end{equation}
\textup{(ii)}   \ 
\[ p_{0}(t, y) := \frac{y+1}{y+b_{1}(t)^{-1}}\text{ and } p_{1}(t, y) := 1-p_{0}(t, y),  \ y \in [\alpha(t), \beta(t)].  \]  
\textup{(iii)}  Let $p_{min}(t)$ and $p_{\text{max}}(t)$ be the minimum  and maximum of $\left\{p_{0}(t, \alpha(t)), p_{1}(t, \beta(t))\right\}$. \\
\textup{(iv)}   \ 
\[  g_{0}(t, x) := 0 \text{ and }  g_{i}(t, x) := G_{X_{i}(x)}\left(t, g_{i-1}(t, x)\right), \ \ x \in [0,1),  i \ge 1.\]  
\textup{(v)}  
\begin{equation*}
H_{n}(t, x) := p_{X_{n+1}(x)}\left(t, g_{i}(t, x)\right),  \ \ P_{n}(t, x) := p_{0}\left(t, g_{n}(t, x)\right),  \ \ x \in [0, 1).   
\end{equation*} 
\textup{(vi)} 
\[ M_n(t, x) := \prod_{i=0}^{n-1} H_{i}(t, x),   \ \  x \in [0, 1). \]
\end{Def}

\begin{Exa}
If $c_0(t) = c_1(t) = 0$ then for $x \in [0,1)$ 
\[ G_0(t, x) = b_1(t) x, \ \ G_1(t,x) = (1-b_1(t))x.  \]
\[ \alpha(t) = g_{n}(t, x) = \beta(t) =  0, \ \ n \ge 0. \] 
\[ p_{0}(t, 0) =  P_n(t, x) = b_1(t) = 1- p_1(t,0), \ \ n \ge 0. \] 
\[ p_{min}(t) = \min\{b_1(t), 1-b_1(t)\} \text{ and } p_{max}(t) = \max\{b_1(t), 1-b_1(t)\}. \] 
\[H_n(t,x) = b_1(t)1_{\{X_{n+1}(x) = 0\}}(x) + (1-b_1(t))1_{\{X_{n+1}(x) = 1\}}(x).  \] 
\[ M_n(t, x) = b_1(t)^{a_{n, 0}}(1-b_1(t))^{n-a_{n,0}} \text{ where }  a_{n,0} := \left|\{1 \le i  \le n : X_i(x) = 0\} \right|.\]    
\end{Exa}

In this case we do not need to introduce $G, g, p, P, H \text{ and } M$.   
However we would like to consider the case that $c_0(t) = c_1(t) = 0$ {\it fails}.  
$G_i, g_n, p_i, H_n \text{ and } M_n$ are defined in order to give a useful expression for $F(t,x)$ in (\ref{Lem1-2}) below.  

The following are easy to see so the details are left to readers. 
\begin{Lem}\label{Basic-Lem}
For $n \ge 0$ and $x \in [0,1)$ \\
\textup{(i)}  \ 
\begin{equation}\label{Basic1} 
\alpha(t) \le g_{n}(t, x) \le \beta(t),\ \    
\end{equation} 
\textup{(ii)}  \ 
\begin{equation}\label{Basic2} 
0 < p_{\text{min}}(t) \le H_{n}(t, x) \le p_{\text{max}}(t) < 1. 
\end{equation} 
\textup{(iii)}  \ 
\begin{equation}\label{Lem1-1}
\mu_{t}\left( [x_{n}, x_{n} + 2^{-n}) \right) = M_n(t,x).  
\end{equation} 
\textup{(iv)}  
\begin{equation}\label{Lem1-2} 
F(t,x) = \sum_{n = 0}^{+\infty} X_{n+1}(x) \left(M_n(t,x) - M_{n+1}(t,x)\right). 
\end{equation} 
\end{Lem}

By (i) $g_n(t,x), H_n(t,x), P_n(t,x)$ and $M_n(t,x)$ are well-defined for any $n$ and $x$.

Define $(\widetilde b_{1}(t), \widetilde c_{0}(t), \widetilde c_{1}(t))$ and  
$(\widetilde A_0(t), \widetilde A_1(t))$ 
by substituting $(b_{1}(t), c_{0}(t), c_{1}(t))$ in Definition \ref{Dual-Def}. 
By (\ref{Double})   
$(\widetilde b_{1}(t), \widetilde c_{0}(t), \widetilde c_{1}(t))$ is also a smooth curve in $E$.   
Define $\widetilde \mu_t, \widetilde F, \widetilde G_{j}, \widetilde g_n, \widetilde p_{j}, \widetilde P_i, \widetilde H_i , \widetilde M_n, \widetilde p_{\text{min}}$ and $\widetilde p_{\text{max}}$ in the same manner 
by substituting $(\widetilde b_0, \widetilde c_0, \widetilde c_1)$ for $(b_0, c_0, c_1)$. 
Lemma \ref{Basic-Lem} hold also for $\widetilde g_i, \widetilde H_n, \widetilde M_n, \widetilde \mu_t, \widetilde p_{\text{min}}$ and $\widetilde p_{\text{max}}$.

\subsection{Well-definedness and continuity of $f_k$}

(\ref{Ass}) has been introduced in order to establish a uniform boundedness for $\partial_{t}^{k} H_n(t,x)$ as follows.    

 \begin{Lem}\label{Bdd}
For any $k \ge 0$ 
there is  a continuous function $C_{1,k}(t)$ on a neighborhood of $t=0$ 
such that for each $t$ in the neighborhood :    
\begin{equation*}
\sup_{n \ge 0, x \in (0,1)} \left| \partial_{t}^{k} H_n(t,x) \right| \le C_{1,k}(t) 
\end{equation*} 
\end{Lem}

\begin{proof}
The case $k=0$ follows from (\ref{Basic2}).  
Assume $k \ge 1$. 
Then $\left| \partial_{t}^{k} H_n(t,x) \right| = \left| \partial_{t}^{k} P_n(t,x) \right|$.   

Recall (\ref{Ass}). 
Assume  $(1+c_1)\left(1-b_1(1+c_0)\right) < 1 - b_1$. 
Then 
\begin{equation}\label{ass-1} 
(1+c_1(t))\left(1-b_1(t)(1+c_0(t))\right) < 1 - b_1(t)
\end{equation}
holds if $t$ is close to $0$. 

Since $\partial_{t}^{k} P_{n}(t, x)$ is a multivariate polynomial consisting of 
\[ \text{$\partial_{t}^{j} g_{n}(t, x)$ and $\partial_{t}^{j^{\prime}}\partial_{y}^{j^{\prime\prime}} P_{n}(t,x)$, 
$0 \le j, j^{\prime}, j^{\prime\prime} \le i$} \] 
as variables,  
it suffices to show that 
for each $k \ge 1$ 
there is  a continuous function $C_{2,k}(t)$ such that 
\begin{equation}\label{Bdd-g}
\sup_{n \ge 0, x \in (0,1)} \left|\partial_{t}^{k}g_{n}(t,x)\right| \le C_{2,k}(t) < +\infty.
\end{equation}  

We now show  (\ref{Bdd-g}) by induction on $k$.    
The case $k=0$ follows from (\ref{Basic1}).   
Assume  (\ref{Bdd-g}) holds for $k = 0, 1, \dots i-1$. Then 
\[ \partial_{t}^{i}g_{n}(t,x) = \partial_{y}G_{X_{n}(x)}(t, g_{n-1}(t,x)) \partial_{t}^{i}g_{n-1}(t,x) + \text{Poly}(i,n). \]
Here $\text{Poly}(i,n)$ is a multivariate polynomial consisting of 
\[ \text{ $\partial_{t}^{j}g_{n-1}(t,x)$ and 
$\partial_{t}^{j^{\prime}}\partial_{y}^{j^{\prime\prime}}G_{X_{n}(x)}(t, g_{n-1}(t,x))$, 
$0 \le j, j^{\prime}, j^{\prime\prime} \le i-1$ } \]
as variables.   

By the hypothesis of induction and (\ref{Basic1}), 
for each $i$, 
there is  a continuous function $C_{3,i}(t)$ such that  
\[ \left|\partial_{t}^{i}g_{n}(t,x)\right| \le \left(\max_{l \in \{0,1\}, y \in [\alpha(t), \beta(t)]}\partial_{y}G_{l}(t, y)\right) \left|\partial_{t}^{i}g_{n-1}(t,x)\right| + C_{3,i}(t). \]
By (\ref{ass-1}) 
\[ \max_{l \in \{0,1\}, y \in [\alpha(t), \beta(t)]} \partial_{y}G_{l}(t, y) < 1.\]  
Therefore  
\[ \sup_{n \ge 0, x \in (0,1)} \left|\partial_{t}^{i}g_{n}(t,x)\right| 
\le \frac{C_{3,i}(t)}{1- \max_{l \in \{0,1\}, y \in [\alpha(t), \beta(t)]} \partial_{y}G_{l}(t, y)}. \] 
Hence (\ref{Bdd-g}) holds.      

Second assume  $(b_1 + c_1)^2 < b_1(1+c_0)(1+c_1)$.
By (\ref{dual-Def}) and continuity of $(b_1(t), c_0(t), c_1(t))$ 
\[ (1+\widetilde c_1(t))\left(1- \widetilde b_1(t)(1+ \widetilde c_0(t))\right) < 1 -  \widetilde b_1(t) \] holds if $t$ is close to $0$. 
The rest of the proof goes in the same manner as above. 
\end{proof}

Let \[ x_{n} := \sum_{i=1}^{n} \frac{X_{i}(x)}{2^i}, \ \ x \in [0,1)\ \text{ and } \  D := \bigcup_{n \ge 1} \left\{\frac{k}{2^{n}} \ \bigg| \   1 \le k \le 2^{n}-1 \right\}.\]

\begin{Thm}\label{M1}
\textup{(i)} For any $k \ge 0$ there is  $C_k > 0$ such that 
\begin{equation}\label{M1-1}
\left|\frac{\partial_{t}^{k} F(0, x_n + 2^{-n}) - \partial_{t}^{k} F(0, x_n)}{F(0, x_n + 2^{-n}) - F(0, x_n)}\right| 
\le C_k n^{k}, \, x \in [0,1), n \ge 1. 
\end{equation}  
\textup{(ii)} $\partial_{t}^{k} F(0, x)$ is well-defined for any $x \in [0,1] \setminus D$. \\ 
\textup{(iii)} Let $C_k$ be the constant above. 
Then 
\begin{equation}\label{M1-C} 
\left| \frac{\partial_{t}^{k} F(0, x) - \partial_{t}^{k} F(0, y)}{F(0, x) - F(0, y)}\right| \le C_k  (-\log_2 |x-y|)^{k},  \,  x \ne y.  
\end{equation}  
\end{Thm}

Now we can define 
\[ f_k(x) := \partial_{t}^{k} F(0, x) \text{ and } \Delta_k F(x, y) := \frac{\partial_{t}^{k} F(0, x) - \partial_{t}^{k} F(0, y)}{F(0, x) - F(0, y)}, \, \,  x \ne y, \, k \ge 0.  \] 

By (\ref{M1-C}), $f_k$ is continuous and  if $\mu_{0}$ is absolutely continuous 
\begin{equation}\label{M1-C-G}  
\left|f_k(x) - f_k(y)\right| = O\left( |x-y|  \left(-\log_{2}|x-y|\right)^{k} \right).   
\end{equation}     
Whether (\ref{M1-C}) is best or not will be discussed after Theorem \ref{D1-new}. 
The key of the proof of (i) is giving an upper bound for $\left|\partial_{t}^{l} H_{j}(t, x) \right|$ uniform with respect to $x$ by Lemma \ref{Bdd}.   
For (ii), roughly speaking, the key is showing the exchangeability of the differential $\partial_t$ with the infinite sum in (\ref{Lem1-2}), by using (\ref{Basic2}). 
(iii) follows from (i) and (ii) easily.

\begin{proof}
By (\ref{Lem1-1})  
\[ \frac{\partial_{t}^{k} F(t, x_{n}+2^{-n}) - \partial_{t}^{k} F(t, x_{n})}{F(t, x_{n}+2^{-n}) - F(t, x_{n})} 
= \frac{\partial_{t}^{k} M_{n}(t, x)}{M_{n}(t, x)}. \] 

There exist positive integers $\left\{r(k, (k_{j})_{j}) : \sum_j k_j = k, k_j \ge 0 \right\}$ such that 
\begin{equation}\label{r} 
\sum_{k_{j} \ge 0, \sum_{j = 0}^{n-1} k_{j} = k} r(k, (k_{j})_j) = n^{k} \ \ \ \text{ and } 
\end{equation}  
\[ \partial_{t}^{k} M_{n}(t, x) = \sum_{k_{j} \ge 0, \, \sum_{j = 0}^{n-1} k_{j} = k} C(k, (k_{j})_{j})
\left(\prod_{j = 0}^{n-1} \partial_{t}^{k_{j}} H_{j}(t, x)\right). \]

We now compare $\partial_{t}^{k_{j}} H_{j}(t, x)$ with $H_{j}(t, x)$.  
Since the number of $j$ such that $k_{j} > 0$ is less than or equal to $k$, 
\[ 
\left|\prod_{j = 0}^{n-1}  \frac{\partial_{t}^{k_{j}} H_{j}(t, x)}{ H_{j}(t, x)}\right| 
= \left| \prod_{j : 0 < k_{j} \le k} \frac{\partial_{t}^{k_{j}} H_{j}(t, x)}{ H_{j}(t, x)} \right|
\le \left(\frac{ \max_{0 \le l \le k, j \ge 0, x \in [\alpha(t), \beta(t)]} \left|\partial_{t}^{l} H_{j}(t, x) \right|}{\min_{j \ge 0, x \in [\alpha(t), \beta(t)]} H_{j}(t, x)}\right)^{k}. 
\]

Lemma \ref{Bdd} implies  for each $l \ge 0$   
\[ \max_{j \ge 0, x \in [\alpha(t), \beta(t)]} \left|\partial_{t}^{l} H_{j}(t, x) \right| \le C_{1,l}(t) < +\infty. \]
By (\ref{r}) 
\begin{equation}\label{M-Fund}  
\left| \frac{\partial_{t}^{k}M_{n}(t, x)}{M_{n}(t,x)} \right|
\le \sum_{k_{j} \ge 0, \ \sum_{j = 0}^{n-1} k_{j} = k} r(k, (k_{j})_{j}) C_{4,k}(t) = C_{4,k}(t) n^{k}, 
\end{equation}  
where $C_{4,k}(t) := \max_{0 \le l \le k} C_{1,l}(t)$.   
This is continuous with respect to $t$.   
Thus we have (i). 

By (\ref{M-Fund}) and (\ref{Basic2})  
there is an open interval $(a, b)$ containing $0$
such that \[ \max_{t \in [a,b]} C_{4,k}(t) < +\infty, \ \ \max_{t \in [a, b]} p_{\text{max}}(t) < 1 \ \ \text{ and}\]   
\begin{equation}\label{upper} 
\sum_{n} \max_{t \in [a,b]} \left| \partial^k_t F(t,x_{n+1}) - \partial^k_t F(t,x_{n}) \right| 
\le \max_{t \in [a,b]} C_{4,k}(t) \cdot \sum_{n \ge 0} n^k \left(\max_{t \in [a, b]} p_{\text{max}}(t)\right)^n. 
\end{equation}  
Recall (\ref{Lem1-2}).  
Thus we have (ii).

(\ref{upper}) implies $\left| f_k(x) - f_k(y) \right| = \lim_{n \to \infty} \left| f_k(x_n) - f_k(y_n) \right|$. 
This and (\ref{M1-1}) imply (\ref{M1-C}). 
The continuity of $\partial^k_t F(0,x)$ with respect to $x$ 
follows from (\ref{M1-C}) and the continuity of $F(x)$.  
Thus we have (iii).
\end{proof} 

Hereafter, if $t = 0$ we often omit $t$ and write $F(x) = F(0,x)$ and $\widetilde F(x) = \widetilde F(0,x)$.

\section{Hausdorff dimension} 

\begin{Thm}\label{Haus}
For any $k \ge 1$, the Hausdorff dimension of the graph of $f_k$ is $1$. 
\end{Thm}

This extends \cite[Corollary 4.2]{AK1} and is applicable to the framework in \cite[Section 5]{dA1}.    
If $f_1 = T$, 
this follows from \cite{MW}. 
For proof we will choose a ``good" family of coverings of the graph of $f_k$ and then show $\dim_{H} \{(x, f_k(x)) : x \in [0,1] \} \le s$ for any $s > 1$.
The key point is using the simple fact that $F$ is the distribution function of $\mu_0$.
Our proof is different from \cite{MW} and \cite{AK1} and seems simpler than them because we do not need to investigate strength of continuity of $f_k$ such as (\ref{M1-C-G}) and the H\"older exponent.    
As we will see in Theorem \ref{D2} (ii) later    
$f_k$ may not be $\eta$-H\"older continuous if $\eta < 1$ is sufficiently close to $1$.

\begin{proof}
Hereafter, ``$\dim_{H}$" denotes the Hausdorff dimension and ``$\text{diam}$" denotes the diameter. 
It is easy to see $\dim_{H} \{(x, f_k(x)) : x \in [0,1]\} \ge 1$.   
We now show  $\dim_{H} \{(x, f_k(x)) : x \in [0,1] \} \le s$ for any $s > 1$.  
Let 
\[ O\left(f_k, n, l\right) := \max_{x \in [(l-1)/2^{n}, l/2^{n}]} \left|f_k(x) - f_k\left(\frac{l-1}{2^{n}}\right)\right| \text{ and }\]
\[ R(k; n,l) :=  \left[\frac{l-1}{2^{n}}, \frac{l}{2^{n}}\right] \times \left[f_k\left(\frac{l-1}{2^{n}}\right) - O\left(f_k, n, l\right), \, f_k\left(\frac{l-1}{2^{n}}\right) + O\left(f_k, n, l\right)\right]. \]  

Then $\cup_{l=1}^{2^n} R(k; n,l)$ covers the graph of $f_k$ and 
\[ \text{diam}\left(R(k; n,l)\right) = (4^{-n} + 4O\left(f_k, n, l\right)^2)^{1/2}.\]    
If $s > 1$,  
\[ \left(4^{-n} + 4O\left(f_k, n, l\right)^2\right)^{s/2} \le \left(2^{-n} + 2O\left(f_k, n, l\right)\right)^{s} \le 2^{s-1} \left(2^{-sn} + 2^{s}O\left(f_k, n, l\right)^s\right). \]
Therefore it suffices to show that
\begin{equation}\label{Osc1}
\lim_{n \to \infty} \sum_{l = 1}^{2^n} O\left(f_k, n, l\right)^s = 0. 
\end{equation}  
By (\ref{M1-C}) 
\[  O\left(f_k, n, l\right) \le C_k \max_{x \in [(l-1)/2^{n}, l/2^{n}]} \left(-\log_2 \left|x-\frac{l-1}{2^{n}}\right| \right)^{k} \left(F(x) - F\left(\frac{l-1}{2^{n}}\right)\right). \]
Using this and \[ \displaystyle \sum_{l = 1}^{2^n} F\left(\frac{l}{2^{n}}\right) - F\left(\frac{l-1}{2^{n}}\right) = 1,\]     
\[ \sum_{l = 1}^{2^n} O\left(f_k, n, l\right)^s 
\le C_k^s \max_{x, y \in [0,1], 0 < |x-y| \le 2^{-n}} \left(-\log_2 |x-y|\right)^{sk} \left|F(x) - F(y)\right|^{s-1}.  \]

Let $z< w$ and $n= n_{z, w}$ be the smallest number $n$ such that $z \le (k-1)/2^n < k/2^n \le w$ for some $k$.   
Then $z \ge \min\{0, (k-3)/2^{n}\}$ and $w \le \max\{1, (k+2)/2^n\}$.   
By (\ref{Basic2}), $p_{\text{max}}(0) < 1$ and  $\max_{v \in (0,1)} \mu_0 ([v_n, v_n + 2^{-n})) \le p_{\text{max}}(0)^n$.   
Hence 
\[ F(y) - F(x) \le F\left(\max\{1, (k+2)/2^n\}\right) - F\left(\min\{0, (k-3)/2^{n}\}\right) \le 5 p_{\text{max}}(0)^n. \] 
Hence 
\begin{equation*}
|F(z) - F(w)| \le 5|z-w|^c, \, \, z, w \in [0,1], \ \ \text{  for $c = -\log_2 p_{\text{max}}(0) > 0$. }   
\end{equation*} 

Using this and $s > 1$,     
\[ \lim_{n \to \infty} \max_{x, y \in [0,1], 0 < |x-y| \le 2^{-n}} \left(-\log_2 |x-y|\right)^{sk} \left|F(x) - F(y)\right|^{s-1} = 0. \]
Thus we have (\ref{Osc1}).  
\end{proof}

\section{Local H\"older continuity at almost every points}

\begin{Thm}\label{SinguDiff} 
There is $c \ge 1$ such that 
for any $k \ge 0$ there is  $C_k^{\prime} < +\infty$ such that 
\[ \limsup_{h \to 0} \frac{|f_k(x+h) - f_k(x)|}{|h|^c} \le C_k^{\prime} \, \, \text{ Lebesgue-a.e.$x$.} \]   
If $\mu_{0}$ is singular, $c > 1$ and $C_k^{\prime} = 0$ for any $k$.    
If $\mu_0$ is absolutely continuous, $c=1$.  
\end{Thm}

This is more general than \cite[Theorems 12 and 13]{dA1} which investigates the case $(b_1(t), c_0(t), c_1(t)) = (t+a,0,0)$, $a \ne 1/2$ only. 
Our approach is partly similar to the proof of \cite[Theorem 12]{dA1} but seems more general and clearer than it.  
The key point is showing the following : 
(1) Giving a nice upper bound for $|F(x) - F(y)|$ in terms of $M_{m}(0, x)$ by (\ref{Lem1-2}) and (\ref{Basic2}).       
(2) $M_{m}(0, x)$ decays rapidly by  (\ref{Lem1-Leb}) below. 
(3) Giving a nice lower bound for $|x-y|$ by assuming $x$ is a normal number as the proof of \cite[Theorem 12]{dA1}.

Let \begin{equation}\label{Def-m} 
\left\{m_{1}(z) < m_{2}(z) < \cdots \right\} := \left\{i \ge 1 : X_{i}(z) = 1 \right\}, \ \ z \in [0,1). 
\end{equation}    
\[ n(x, y) := \min\left\{n : m_{k}(x) = m_{k}(y)  \text{ for any } k \le n \right\}, \ \ x, y \in (0,1) \setminus D \text{ with } x \ne y.  \]

In a manner similar to the proof of  \cite[Theorem 1.2]{O1}\footnote{(\ref{Lem1-Leb}) is a statement for the {\it Lebesgue measure}. 
Hence   
we need to alter the arguments in the proofs of \cite[Lemma 2.3 (2) and Lemma 3.3]{O1} slightly. 
Since the alteration is easy we omit the details.},   
there is a constant $c \ge 1$ such that 
\begin{equation}\label{Lem1-Leb}
\liminf_{n \to \infty} \frac{-\log_2 M_n(0, x)}{n} \ge c \ \  \text{ Lebesgue-a.e.$x$.}  
\end{equation}
If $\mu_0$ is singular, $c > 1$.  
If $\mu_0$ is absolutely continuous, $c=1$.

\begin{proof} 
This assertion is trivial if $\mu_0$ is absolutely continuous. 
Assume $\mu_0$ is singular and $x$ is a normal number. 
Using (\ref{M1-C}), it suffices to show 
\begin{equation}\label{SD1} 
\lim_{h \to 0} \frac{|F(x+h) - F(x)|}{|h|^c} = 0  \ \ \ \text{ Lebesgue-a.e.$x$ \ for some $c > 1$.} 
\end{equation}
Let $y \in (0,1) \setminus (D \cup \{x\})$ and Let $m := m_{n(x, y)}(x)$.  
Then $X_k(x) = X_k(y)$ and  
$M_{k}(0, x) = M_{k}(0, y)$ for any $k \le m_{n(x,y)}(x)$. 
By (\ref{Lem1-2}) and (\ref{Basic2}), 
\begin{align}\label{SD2} 
\left|F(x) - F(y)\right| \le \sum_{i \ge m} \left| M_{i}(0, x) -  M_{i}(0, y) \right|  \le C M_{m}(0,x).    
\end{align} 
Here $C$ denotes a constant independent from $x, y$.  

We now give a lower bound of $|x-y|$ in terms of $n(x, y)$.  
If $x > y$,   $m_{n(x,y)+1}(x) < m_{n(x,y)+1}(y)$ and hence 
\[ x - y \ge 2^{-m_{n(x,y) + 2}(x)}.\]     
If $x < y$, $m_{n(x,y)+1}(x) > m_{n(x,y)+1}(y)$ and hence 
\[ y - x \ge 2^{-m_{n(x,y)+1}(x)}  - \sum_{j \ge n(x,y) + 2} 2^{-m_{j}(x)}.\]

Since $x$ is normal, 
\begin{equation}\label{xy}
|x-y| \ge 2^{-m_{n(x, y) + 2}(x) \cdot (1+o(1))}, \, y \to x. 
\end{equation}

By (\ref{Lem1-Leb}) and $\displaystyle\lim_{k \to +\infty} \dfrac{m_{k+2}(x)}{m_k(x)} = 1$, 
there is $c>1$ such that 
\begin{equation}\label{Lem1-Leb-2} 
\lim_{k \to \infty} 2^{c \cdot m_{k+2}(x)} M_{m_k(x)}(0, x) = 0 
\end{equation}   
holds for Lebesgue-a.e. normal number $x$.   
We also have 
$\lim_{y \to x} n(x, y) = +\infty$.   
Now (\ref{SD1}) follows from (\ref{SD2}), (\ref{xy}) and (\ref{Lem1-Leb-2}).          
\end{proof}


\section{Asymptotics of $f_k$ around dyadic rationals} 

\subsection{Lemmas}

Recall the definition of $g_i$, $P_i$ and $H_i$ in Definition \ref{Def-many}.  
Then 
\begin{equation}\label{Lem2-1} 
\partial_t P_{i}(t, x) 
= \frac{b_{1}^{\prime}(t)(g_{i}(t, x)+1) + b_{1}(t)(1-b_{1}(t))\partial_{t}g_{i}(t, x)}{\left(b_{1}(t)g_{i}(t, x) + 1\right)^2}.  
\end{equation} 

Let \[ D_n := \left\{\frac{k}{2^{n}} : 1 \le k \le 2^{n}-1 \right\}, \ \ n \ge 1 \ \text{  and }  \ D_0 := \emptyset. \]

\begin{Lem}\label{H-exp-fast}
\textup{(i) }\begin{equation}\label{H-exp-fast-right}
\lim_{i \to \infty} \sup_{y \in D_k \setminus D_{k-1}, k \ge 1} \left| H_{i+k}(0, y)  - b_1(1+c_0)\right| = 0.  
\end{equation} 
This also holds if we substitute $\widetilde H_{i+k}, \widetilde b_1 \text{ and } \widetilde c_0$ for $H_{i+k}, b_1, \text{ and } c_0$.\\     
\textup{(ii) } If $x \in D$,    
\begin{equation}\label{H-exp-fast-partial} 
\lim_{n \to \infty} \frac{\partial_t H_{n}(0,x)}{H_n(0, x)} = \frac{b_1^{\prime}(0)}{b_1} + \frac{c_0^{\prime}(0)}{1+c_0} 
\end{equation}   
Convergences (\ref{H-exp-fast-right}) and (\ref{H-exp-fast-partial}) are exponentially fast. 
\end{Lem}

\begin{proof}
Recall the definition of $G_i$ in (\ref{Def-G}).    
Let $G_{0, i}$ be the $i$-th composition of $G_{0}(0, \cdot)$.  
Since the Lipschitz constant of $G_0(0, \cdot)$ on $[\alpha, \beta]$ is strictly smaller than $1$,   
\begin{equation}\label{G-conv}
\lim_{i \to \infty} \sup_{z \in [\alpha, \beta]} \left| G_{0, i}(z) - \frac{c_0}{1-b_1(1+c_0)} \right| = 0.  
\end{equation}    
This convergence is exponentially-fast. 
If $y \in D_k \setminus D_{k-1}$, $H_{i+k}(0, y) = p_0\left(0, G_{0, i}(g_k(0, y))\right)$.  
Hence (\ref{H-exp-fast-right}) holds and the convergence is exponentially fast. 
Since the Lipschitz constant of $\widetilde G_0(0, \cdot)$ on $\left[\widetilde \alpha, \widetilde \beta\right]$ strictly smaller than $1$,   
(\ref{H-exp-fast-right}) holds for $\widetilde H_{i+k}, \widetilde b_1 \text{ and } \widetilde c_0$.   
Thus we have (i).    

We have 
\begin{equation}\label{eq-g} 
\partial_{t} g_{i}(t, x) 
= \partial_{t} G_{X_{i}(x)}\left(t, g_{i-1}(t, x)\right) + \partial_{y} G_{X_{i}(x)}\left(t, g_{i-1}(t, x)\right) \partial_{t} g_{i-1}(t, x). 
\end{equation} 

Note $X_n(x) = 0$ for large $n$.  
By (\ref{G-conv}) and (\ref{eq-g}) 
\begin{equation*}
\lim_{n \to \infty} g_n(0,x) = \frac{c_0}{1-b_1(c_0 + 1)} \text{ \ \ \ exponentially fast \  and }   
\end{equation*}
\begin{equation*}
\lim_{n \to \infty} \partial_t g_n(0,x) = \frac{\partial_{t}G_{0}\left(0, \frac{c_{0}}{1-b_{1}(c_{0}+1)}\right)}{1 - \partial_{y}G_{0}\left(0, \frac{c_{0}}{1-b_{1}(c_{0}+1)}\right)} \text{ \ \ exponentially fast.}   
\end{equation*}
Using these convergences, (\ref{Lem2-1}) and  (\ref{H-exp-fast-right}), we have (ii).   
\end{proof}

\subsection{Non-degenerate condition}

If all of $b_{1}(t), c_{0}(t)$ and $c_{1}(t)$ are constant, $f_k(x) = 0$ for {\it any} $x \in [0,1]$ and $k \ge 1$.  
In this case the estimate in (\ref{M1-C}) is not best.  
We now introduce a ``non-degenerate" condition for the curves and consider the estimate in (\ref{M1-C}) is best or not under the condition. 

\begin{Def}[A non-degenerate condition]\label{ND}
We say {\it (ND) holds}  
if 
\begin{equation}\label{ND-1}
b_{1}^{\prime}(0)(\alpha +1) + b_{1}(1-b_{1}) \min\{0, \delta_{0}, \delta_{1}\} > 0  
\text{ or }
\end{equation} 
\begin{equation}\label{ND-2}
(\widetilde b_{1})^{\prime}(0)(\widetilde\alpha +1) + \widetilde b_{1}(1-\widetilde b_{1})\max\{0, \widetilde\delta_{0}, \widetilde\delta_{1}\} < 0,  
\end{equation}  
\[ \text{where } \ \ \delta_{i} := \min_{y \in \left[\alpha, \beta\right]}  \frac{\partial_{t}G_i(0, y)}{1- \partial_{y}G_i(0, y)}, \,  \, 
\widetilde\delta_{i} := \max_{y \in \left[\widetilde\alpha, \, \widetilde\beta \right]}  \frac{\partial_{t} \widetilde G_i(0, y)}{1- \partial_{y}\widetilde G_i(0, y)},  \, \, i=0,1.\]
Recall (\ref{Def-G}) for the definitions of $G_i$ and $\widetilde G_i$.   
Both $\delta_0$ and $\widetilde \delta_0$ are well-defined. 
On the other hand either $\delta_{1}$ or $\widetilde \delta_{1}$ is well-defined. 
\end{Def}

By this condition the derivative of $F(t, x_n + 2^{-n}) - F(t, x_n)$ with respect to $t$ is positive at $t=0$. 
In particular $f_k$ is {\it not} a constant. 
See Lemma \ref{ND-Lemma} for details.  
If $\gamma(t)$ is a smooth curve with $\gamma(0) = 0$ and $\gamma^{ \prime}(0) > 0$, 
(ND) holds for $(b_{1}(t), c_{0}(t), c_{1}(t))$ 
if and only if it also holds for $\left(b_{1}(\gamma(t)), c_{0}(\gamma(t)), c_{1}(\gamma(t))\right)$.   

This condition is somewhat complex. 
However (ND) holds for $T$ and its generalizations in \cite{AK1},  \cite{AK-JMAA}, \cite{dA1} etc.
If $c_{1}(t) = c_{2}(t) = 0$ for any $t$, $\delta_{0}(t) = \delta_{1}(t) = \widetilde\delta_{0}(t) = \widetilde\delta_{1}(t) = 0$ and hence (ND) holds if and only if $b_{1}^{\prime}(0) > 0$.
Hereafter we will not use (ND) explicitly.  
Instead the following will be used.   

\begin{Lem}\label{ND-Lemma} 
If  (ND) holds, 
\begin{equation}\label{ND-Lem-1}
\inf_{n \ge 0, x \in [0,1]}  \partial_t P_{n}(0, x) > 0.  
\end{equation} 
\end{Lem}

\begin{proof}
First assume (\ref{ND-1}). 
Recall (\ref{eq-g}). 
If $\eta \le \min\{\delta_{0}, \delta_{1}\}$ and $\partial_{t} g_{i-1}(0, x) \ge \eta$ then \[ \partial_{t} g_{i}(0, x) \ge \eta.\] 
Since $\partial_{t} g_{0}(0, x) = 0$, 
\[ \partial_{t} g_{i}(0, x) \ge \min\left\{0, \delta_{0}, \delta_{1}\right\},   i \ge 0.\]  
Using (\ref{Basic1}), (\ref{Lem2-1}) and  (\ref{ND-1}), 
we have (\ref{ND-Lem-1}).  

Second assume (\ref{ND-2}).     
Recall (\ref{Dual}). Then   
\begin{equation*}
\widetilde P_k(t, x) = 1 - P_k(t, 1-x), \, \, x \in (0,1) \setminus D_k, \, k \ge 1. 
\end{equation*}
\begin{equation*}
\sup_{n \ge 0, x \in [0,1)} \partial_t \widetilde P_{n}(0, x) < 0.  
\end{equation*} 
(\ref{ND-Lem-1}) follows from these claims. 
\end{proof}

\subsection{Comparing $ \Delta_k F(x, x+h)$ with $(\log_2 (1/|h|))^k$ at dyadic rationals}

\begin{Thm}\label{D1-new}
For any $k \ge 0$ and any $x \in D$, 
\begin{equation}\label{D1-new-1} 
\lim_{h \to 0, h > 0} \frac{ \Delta_k F(x, x+h)}{(\log_2 (1/|h|))^k} = \left( \frac{b_1^{\prime}(0)}{b_1} + \frac{c_0^{\prime}(0)}{c_0 + 1} \right)^{k}. 
\end{equation} 
\begin{equation}\label{D1-new-2} 
\lim_{h \to 0, h < 0}  
\frac{ \Delta_k F(x, x+h)}{(\log_2 (1/|h|))^k} = \left(- \frac{b_1^{\prime}(0)}{b_1} - \frac{c_1^{\prime}(0)}{c_1 + 1} \right)^{k}. 
\end{equation}  
\end{Thm}

If (ND) holds, $b_1^{\prime}(0)/b_1 + c_i^{\prime}(0)/(c_i + 1)$, $i=0,1$, above are positive and hence 
we can not replace $\left(-\log_{2}|x-y|\right)^{k}$ with smaller functions in (\ref{M1-C}). 
This extends Kr\"uppel \cite[Proposition 3.2]{Kr}.  
\cite[Theorem 4.1]{AK1} follows from this and (\ref{M1-C}).    

For proof first consider the asymptotic of $\Delta_k F(x, x+2^{-n})$ as $n \to \infty$ as in (\ref{induction}) below. 
We will show this by induction on $k$ and Lemma \ref{H-exp-fast} (ii).  
Then replace $``2^{-n}"$ in $\Delta_k F(x, x+2^{-n})$ with $h > 0$.

\begin{Def}
Let 
\begin{equation}\label{Z-Def}
Z_{k, n}(x) :=  \Delta_{k} F(x_n, x_n + 2^{-n}), \, x \in [0,1), n \ge 1, k \ge 0.  
\end{equation}
Define $\widetilde Z_{k, n}$ by substituting $\widetilde F$ for $F$. 
\end{Def}

\begin{Prop}\label{D1}
For any $k \ge 0$ and any $x \in D$,
\begin{equation}\label{D1-1} 
\lim_{n \to \infty} \frac{ \Delta_k F(x, x+2^{-n})}{n^k} = \left( \frac{b_1^{\prime}(0)}{b_1} + \frac{c_0^{\prime}(0)}{c_0 + 1} \right)^{k}. 
\end{equation} 
\begin{equation}\label{D1-2} 
\lim_{n \to \infty}  
\frac{ \Delta_k F(x, x-2^{-n})}{n^k} = \left( - \frac{b_1^{\prime}(0)}{b_1} - \frac{c_1^{\prime}(0)}{c_1 + 1} \right)^{k}. 
\end{equation}  
\end{Prop}

\begin{proof}
Let $x \in D$. 
Then $x = x_{n}$ and $P_{n}(t, x) = H_n(t, x)$ 
hold for any $t$ and sufficiently large $n$.  
We now show  
\begin{equation}\label{induction}
 \lim_{n \to \infty} \frac{Z_{k, n}(x)}{n^{k}} = q_{1}^{k} \ \  \text{ where } q_1 := \frac{b_1^{\prime}(0)}{b_1} + \frac{c_0^{\prime}(0)}{1+c_0}   
\end{equation} 
by induction on $k$. 
The case $k=0$ follows immediately. 
Assume  (\ref{induction})
holds for any $k = 0,1, \dots, l-1$.  
Differentiating 
\[ F(t, x_{n+1}+2^{-n-1}) - F(t, x_{n+1}) = \left(F(t, x_{n}+2^{-n}) - F(t, x_{n})\right) H_n(t, x)  \]
$l$ times with respect to $t$ at $t = 0$,  
\begin{equation}\label{Z-vs-Z} 
Z_{l, n+1}(x) = Z_{l, n}(x) + l \frac{\partial_{t}P_{n}(0, x)}{P_{n}(0, x)} Z_{l-1, n}(x) + \sum_{i=2}^{l} \binom{l}{i} \frac{\partial_{t}^{ \, i}P_{n}(0, x)}{P_{n}(0, x)} Z_{l-i, n}(x). 
\end{equation}

By (\ref{M1-C}) and (\ref{Z-Def}), $Z_{k, n}(x) = O(n^k)$.  
By (\ref{Basic2}) and Lemma \ref{Bdd}, 
\[ \sum_{i=2}^{l} \binom{l}{i} \frac{\partial_{t}^{ \, i} P_{n}(0, x)}{P_{n}(0, x)} Z_{l-i, n}(x) = O\left(n^{l-2}\right). \] 
Using this, (\ref{H-exp-fast-partial}) and  the hypothesis of induction,  
\[ \lim_{n \to \infty} \frac{Z_{l, n+1}(x) - Z_{l, n}(x)}{n^{l-1}} = l \lim_{n \to \infty} \frac{\partial_{t}P_{n}(0, x)}{P_{n}(0, x)} \frac{Z_{l-1, n}(x)}{n^{l-1}} = l q_{1}^{l}. \]
Hence  
(\ref{induction})
holds for $k = l$. 
Thus we have (\ref{D1-1}).

In the same manner as above 
\[ \lim_{n \to \infty}\frac{ \Delta_k F(x-2^{-n}, x)}{n^k} = \lim_{n \to \infty} \frac{\widetilde Z_{k,n}(1-x)}{n^k} =  (- q_2)^k  \ \ 
\text{ where } q_2 := \frac{b_1^{\prime}(0)}{b_1} + \frac{c_1^{\prime}(0)}{c_1 + 1}. \]
Thus we have (\ref{D1-2}). 
\end{proof}

We will show  Theorem \ref{D1-new} using Proposition \ref{D1} crucially.   
Roughly, what we need to show is substituting $h$ for $2^{-n}$ in Proposition \ref{D1}.   
Recall (\ref{Def-m}) for the definition of $\{m_n(z)\}_n$.

\begin{proof}[Proof of Theorem \ref{D1-new}]
Let $x \in D$ and $n_0 := \min\{n : x \in D_n\}$.  
If $m_1(h) > n_0$, 
\[ \Delta_{k} F(x, (x+h)_{m_{1}(h)}) = Z_{k, m_{1}(h)}(x) \] 
and hence 
\[ \Delta_{k} F(x, x+h)  =  Z_{k, m_{1}(h)}(x) \]
\[ + \sum_{i=2}^{\infty} \frac{F((x+h)_{m_{i}(h)}) - F((x+h)_{m_{i-1}(h)})}{F(x+h) - F(x)} \left(Z_{k, m_{i}(h)}((x+h)_{m_{i-1}(h)}) -  Z_{k, m_{1}(h)}(x)\right). \]
By (\ref{Basic2}) and $((x+h)_{m_{i-1}(h)})_{m_1(h) - 1} = x$, 
\[ \frac{F((x+h)_{m_{i}(h)}) - F((x+h)_{m_{i-1}(h)})}{F(x+h) - F(x)} \le \frac{M_{m_{i}(h)}(0, (x+h)_{m_{i-1}(h)})}{M_{m_{1}(h)-1}(0, x)} \le p_{\text{max}}(0)^{m_i(h) - m_1(h)}. \]

Using (\ref{Z-vs-Z}), Lemma \ref{Bdd} and  (\ref{M1-1}), 
there is  a constant $C_{k}^{\prime \prime} < +\infty$ such that 
\[ \left|Z_{k, m_{i}(h)}((x+h)_{m_{i-1}(h)}) - Z_{k, m_{1}(h)-1}(x)\right| \le C_{k}^{\prime \prime} m_1(h)^{k-1} \left(m_i(h) - m_1(h)\right)^{k}.  \]

Therefore 
\[ \frac{\left|\Delta_{k} F(x, x+h)  - Z_{k, m_{1}(h)}(x)\right|}{m_1(h)^{k}} 
\le C_{k}^{\prime \prime} \frac{1}{m_1(h)} \sum_{n \ge 1} n^k p_{\text{max}}(0)^n. \]

The right hand side goes to $0$ as $h \to 0, h > 0$.   
By this and (\ref{D1-1})     
we have (\ref{D1-new-1}).
We can show (\ref{D1-new-2}) in the same manner by using (\ref{D1-2}).    
\end{proof}

The asymptotic of $f_k(x)$ around $x \in D$ are quite different depending on $(b_{1}, c_{0}, c_{1})$.   

\begin{Thm}\label{D2}
Let $x \in D$. Then\\
\textup{(i) } If $c_{0} < (1-2b_{1})/2b_{1}$ and $c_{1} > 1-2b_{1}$, 
there is  $c > 1$ such that  
\begin{equation}\label{D2-1}
\lim_{h \to 0} \frac{|f_k(x+h) - f_k(x)|}{|h|^c}  = 0.  
\end{equation}   
\textup{(ii) } Assume  (ND) holds. 
If $c_{0} \ge (1-2b_{1})/2b_{1}$ or $c_{1} \le 1-2b_{1}$, 
there is  $c \le 1$ such that  
\begin{equation}\label{D2-2}
\limsup_{h \to 0} \frac{|f_k(x+h) - f_k(x)|}{|h|^c (\log_2(1/|h|))^k} = +\infty.  
\end{equation} 
If $\mu_0$ is singular, $c < 1$.  
If $\mu_0$ is absolutely continuous, $c=1$. 
\end{Thm}

(i) is similar to Theorem \ref{SinguDiff} 
and consistent with (\ref{Cond}) and (\ref{Ass}).   
An example of a graph of $f_1$ satisfying $c_{0} < (1-2b_{1})/2b_{1}$ and $c_{1} > 1-2b_{1}$ is given in Figure \ref{fig1} below. 
We will show (i) in a manner similar to the proof of Theorem \ref{SinguDiff}.  
The key point is showing $M_{m}(0, x)$ decays rapidly.  
We will show it by Lemma \ref{H-exp-fast} (i), 
which plays the same role as (\ref{Lem1-Leb}) in the proof of Theorem \ref{SinguDiff}.   
If $\mu_{0}$ is absolutely continuous or $c_0 = c_1 = 0$,  $c_{0} \ge (1-2b_{1})/2b_{1}$ or $c_{1} \le 1-2b_{1}$. 
For the proof of (ii), by Theorem \ref{D1-new}, it suffices to show $|F(x+h) - F(x)| \ge c |h|^c$. 
We will show it by Lemma \ref{H-exp-fast} (i).

\begin{proof} 
Let $x \in D$.  
By Lemma \ref{H-exp-fast} (i) and $b_{1}(c_{0}+1) < 1/2$, 
\[ \lim_{m \to \infty} 2^{cm} m^{k} M_{m}(0, x) = 0 \text{ for some $c > 1$}.  \]
Using this, (\ref{M1-1}) and (\ref{SD2}), 
\begin{equation}\label{D2-right} 
\lim_{h \to 0, h > 0} \frac{|f_k(x+h) - f_k(x)|}{|h|^c} = 0.  
\end{equation} 

Since $c_1 > 1 - 2b_1$, 
$\widetilde b_{1}(\widetilde c_{0}+1) < 1/2$ and 
(\ref{D2-right}) holds also for $\partial_{t}^{k} \widetilde F(0,x)$.
By (\ref{Dual}) 
\begin{equation*} 
\partial_t^k \widetilde F(t, x) = - \partial_t^k F(t, 1-x), \ \ x \in (0,1),  \, k \ge 1.  
\end{equation*}  
Therefore 
\begin{equation}\label{D2-left} 
\lim_{h \to 0, h > 0} \frac{|f_k(x) - f_k(x-h)|}{|h|^c} = 0.   
\end{equation}  
(\ref{D2-right}) and (\ref{D2-left}) imply (\ref{D2-1}).

We now show (ii). 
Assume $c_{0} \ge (1-2b_{1})/2b_{1}$. 
It is equivalent to $b_{1}(c_{0}+1) \ge 1/2$. 
By  Lemma \ref{H-exp-fast} (i), 
for some $c \le 1$ which does not depend on $x$,  
\begin{equation}\label{D2-2-right}
\liminf_{n \to \infty} 2^{c \cdot n} (F(x+2^{-n}) - F(x)) > 0.  
\end{equation}

Assume $c_{1} \le 1-2b_{1}$. 
Then $\widetilde c_{0} \ge (1-2\widetilde b_{1})/2\widetilde b_{1}$. 
Therefore (\ref{D2-2-right}) holds for $\widetilde F$. 
Hence for some $c \le 1$ which does not depend on $x$,  
\begin{equation}\label{D2-2-left}  
\liminf_{n \to \infty} 2^{c \cdot n} (F(x) - F(x-2^{-n})) 
= \liminf_{n \to \infty} 2^{c \cdot n} (\widetilde F(1-x + 2^{-n}) - \widetilde F(1-x))  > 0. 
\end{equation}  

Since either (\ref{D2-2-right}) or (\ref{D2-2-left}) holds, 
\[ \limsup_{h \to 0} \frac{|F(x+h) - F(x)|}{|h|^c} > 0. \] 
If $\mu_0$ is singular, $c < 1$. If it is absolutely continuous, $c = 1$. 
Using this, Lemma \ref{ND-Lemma} and  Theorem \ref{D1-new}, we have (\ref{D2-2}).  
\end{proof}

\begin{figure}[h]
\centering
\includegraphics[width = 60mm, height = 45mm, bb = 0 0 846 594]{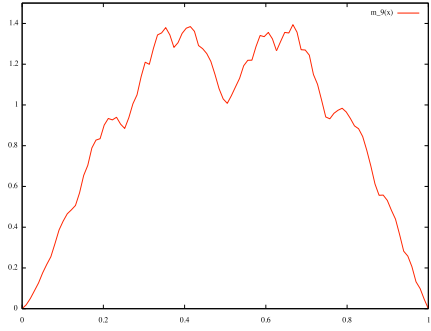}
\caption{\small Graph of $f_1$ for $\displaystyle (b_{1}(t), c_0(t), c_1(t)) = \left(t+\frac{1}{2},  \ -\frac{1}{3}, \ \frac{1}{3}\right)$.}
\label{fig1}
\end{figure}


\section{Results for two special cases}

We say {\it (L) holds} if $(b_1(t), c_0(t), c_1(t)) = (t+a, 0,0)$ for some $a \in (0,1)$. 
{\it In this section we always assume (ND) holds and either $k=1$ or (L) holds.}  
Recall Definition \ref{ND} and Lemma \ref{ND-Lemma} for (ND).     

Let \[ Y_i(x) := \frac{\partial_{t} H_{i-1}(0, x)}{H_{i-1}(0, x)}, \, i \ge 1, \text{ and } Y_0(x) := 0.\]  
$Y_i(x) > 0$ if and only if $X_{i}(x) = 0$. 
If (L) holds, 
\[  Y_i(x) = \frac{1}{a}1_{\{X_i(x) = 0\}}(x) + \frac{1}{1-a}1_{\{X_1(x) = 1\}}(x). \] 

Lemmas \ref{Bdd} and \ref{ND-Lemma} imply
\begin{equation}\label{Fund-Y}
0 < \inf_{i \ge 1, x \in (0,1)} |Y_i(x)| \le \sup_{i \ge 1, x \in (0,1)} |Y_i(x)| < +\infty.    
\end{equation}

Recall the definition of $Z_{k,n}$ in (\ref{Z-Def}). Then  
\begin{equation}\label{ZY}
Z_{1,n}(x) = \sum_{i=1}^{n} Y_i(x). 
\end{equation}  

If (L) holds, using (\ref{Z-vs-Z}), 
\begin{equation}\label{ZZ} 
Z_{k, n+1}(x) -  Z_{k, n}(x) = k Y_{n+1}(x) Z_{k-1, n}(x), \ \  x \in [0,1), \ \ k \ge 2.   
\end{equation} 

Let $\mu_0( \cdot | A)$ be the conditional probability of $\mu_0$ given a Borel measurable set $A$.   
Denote the expectation with respect to $\mu_0( \cdot | A)$ by $E^{\mu_0}_{A}$. 
Let \[ \mathcal{F}_{n} := \sigma\left(\left\{\left[\frac{k}{2^{n}}, \frac{k+1}{2^{n}}\right) \  \bigg| \ 0 \le k \le 2^{n} -1 \right\}\right), \ \ n \ge 0.\]  
Then 
$\{Z_{k,i} \}_{i \ge n}$ is a $\{\mathcal{F}_{i}\}_i$-martingale\footnote{See Williams' book \cite{W} for definition.} with respect to $\mu_0( \cdot | A)$ for $A \in \mathcal{F}_n$. 
By  induction on $k$ 
\begin{equation*}
Z_{k, n} = k! \sum_{1 \le i_1 < \cdots < i_k \le n} \left(\prod_{j=1}^{k} Y_{i_{j}}\right) = O(n^{k}).  
\end{equation*}

\begin{Lem}[\text{Fluctuation of $\{Z_{k, n}\}_n$}]\label{Lem-Z} 
For each $k \ge 1$ 
\begin{equation}\label{strict} 
\limsup_{m \to +\infty} Z_{k,m}(x) > \liminf_{m \to +\infty} Z_{k,m}(x), \ \ x \in (0,1).  
\end{equation} 
\begin{equation}\label{Z-infty}
\limsup_{m \to \infty} E^{\mu_0}_{A}\left[ \left|Z_{k, m}\right| \right] = +\infty, \ \ A \in \mathcal{F}_n,  n \ge 1.   
\end{equation}  
\end{Lem}

\begin{proof}
The case $k=1$ of (\ref{strict}) follows from (\ref{Fund-Y}).  
Assume (L) holds. 
We now show this by induction on $k$.  
Assume that this assertion holds for $k = 1, \dots, l$ and 
\[ \limsup_{n \to +\infty} Z_{l+1, n}(x) = \liminf_{n \to +\infty} Z_{l+1, n}(x) \ \text{ for some $x$.} \] 
Then (\ref{ZZ}) and (\ref{Fund-Y}) imply
$\lim_{n \to \infty} Z_{l, n}(x) = 0$.  
This contradicts the assumption of induction. 
Hence 
\[ \limsup_{n \to +\infty} Z_{l+1, n}(x) > \liminf_{n \to +\infty} Z_{l+1, n}(x)\] for {\it any} $x$.  
Thus we have (\ref{strict}).  
Using this, (\ref{strict}) and the martingale convergence theorem (\cite[Chapter 11]{W}),  we have (\ref{Z-infty}).  
\end{proof}

\subsection{Differentiablity and variation} 

For $g : [0,1] \to \R$ and $0 \le a \le b \le 1$, let    
\[ V(g; [a,b]) := \sup\left\{ \sum_{i=1}^{n} |g(t_{i}) - g(t_{i-1})| \ \ \bigg| \ \ a = t_0 < t_1 < \cdots < t_n = b \right\}. \]  

\begin{Thm}\label{N-Conv-V}
\textup{(i) } $($Non-differentiablity for the absolutely continuous case$)$  
For any $x \in (0,1)$,   
$\Delta_k F(x, x+h)$ does not converge to any real number as $h \to 0$.
In particular, if $\mu_0$ is absolutely continuous, $f_k$ is not differentiable at any point in $(0,1)$.\\
\textup{(ii) } $($Variation$)$ 
$V(f_k; [a, b]) = +\infty$ holds for any $[a,b] \subset [0,1]$, $a < b$.     
\end{Thm}

(i) is an extension of \cite[Theorem 5.1]{AK1}. 
A problem of this kind was also considered by \cite{T}. 
Our proof of (i) is somewhat similar to Billingsley \cite{B} and \cite[Theorem 5.1]{AK1}.  
The key is a fluctuation of $\{Z_{k,n}\}_n$ in (\ref{strict}).  
It seems natural to consider whether $f_k$ is of bounded variation.   
To our knowledge variations of $f_k$ have not been considered. 
The key of proof of (ii) is showing, by using (\ref{Z-infty}), 
the expectation of $\left|Z_{k,m}\right|$ under $\mu_0$ on an interval diverges to infinity.   

\begin{proof}
If $x \in D$, 
the assertion follows from Theorem \ref{D1-new} and the condition (ND).       
Assume  $x \notin D$.  
It is easy to see that for any $k, n \ge 1$ 
\begin{align}\label{minmax}
\min\left\{ \Delta_k F(x, x_n),  \Delta_k F(x, x_n + 2^{-n}) \right\}  
&\le Z_{k,n}(x) \notag \\
&\le \max  \left\{ \Delta_k F(x, x_n),  \Delta_k F(x, x_n + 2^{-n}) \right\}   
\end{align} 
By (\ref{strict}) $Z_{k,n}(x)$ does not converge to any real number. 
Therefore 
if $(Z_{k,n}(x))_n$ diverges as $n \to +\infty$, \[ \limsup_{h \to 0} \left|\Delta_k F(x, x+h)\right| = +\infty. \] 
If $(Z_{k,n}(x))_n$ fluctuates as $n \to +\infty$, 
\begin{align*} 
\limsup_{h \to 0} \Delta_k F(x, x+h) - \liminf_{h \to 0} \Delta_k F(x, x+h) &\ge \limsup_{n \to +\infty} Z_{k,n}(x) - \liminf_{n \to +\infty} Z_{k,n}(x)\\  
&\ge c
\end{align*}  
for some $c = c(x) > 0$.       
These imply (i).  
By (\ref{Z-Def}), 
\[ \sum_{l=2^{m-n}(j-1) + 1}^{2^{m-n}j} \left| f_k\left(\frac{l}{2^m}\right) - f_k\left(\frac{l-1}{2^m}\right) \right| 
= E^{\mu_{0}}_{[(j-1)/2^n, j/2^n)} \left[  \left|Z_{k,m}\right|  \right], \, m > n. \]
(ii)  follows from this and (\ref{Z-infty}).   
\end{proof}

\subsection{MTNI} 

\begin{Thm}[MTNI]\label{MTNI}
For some $c \in [0,1]$ the following hold :\\
\textup{(i) }
\begin{equation}\label{MTNI-k} 
 \ \ \ \ \ \ \ \ \ \ \ \ \ \ \ \  \limsup_{h \to 0} \frac{|f_k(x+h) - f_k(x)|}{|h|^c} = +\infty   \ \ \text{  $\mu_0$-a.s.$x$.}  \ \ \ \ \ \ \ \ \ \ \ \ \ \ \
\end{equation}
\textup{(ii) } For any open interval $J$ 
\[   \ \  \sup_{x, y \in J, x > y}\frac{f_k(x) - f_k(y)}{(x-y)^{c}} = +\infty \ \text{ and }  
\inf_{x, y \in J, x > y}\frac{f_k(x) - f_k(y)}{(x-y)^{c}} = -\infty \ \ \]
If $\mu_0$ is singular, $c < 1$. If $\mu_0$ is absolutely continuous, $c = 1$. 
If $(L)$ holds, $c$ does not depend on $k$.  
\end{Thm}

(i) corresponds to Theorem \ref{SinguDiff} but here the limit diverges.    
If $\mu_0$ is singular, 
the asymptotic of $f_k$ {\it around  Lebesgue-a.e. points} are quite different from the ones {\it around $\mu_0$-a.s. points}.   
(ii) extends \cite[Theorem 14]{dA1}.    
The proof of \cite[Theorem 14]{dA1}\footnote{As far as the author sees, the proof of \cite[Theorem 14]{dA1} seems more complex than the one of \cite[Proposition 6]{dA1}.} is omitted in \cite{dA1}. 
However the reason that the proof of \cite[Proposition 6]{dA1} is not applied to {\it even} $k$ is not described in \cite{dA1}.   
We will give a proof applied to {\it all} $k$ together.   

For the proofs we first compare $\Delta_k F$ with $Z_{k,n}$ by (\ref{minmax}) and then estimate $F(x_n + 2^{-n}) - F(x_n)$ by (\ref{Lem1-mu}) below. 
For (i) we will give a lower bound for $\left|f_k(x_n + 2^{-n}) - f_k(x_n)\right|$ in terms of $|Z_{k, n}|$.  
Remark that $|Z_{k, n}|$ is positive by (\ref{strict}).
For (ii), by probabilistic techniques we will choose $x$ such that $f_k(x_n + 2^{-n}) - f_k(x_n)$ is ``larger" than the positive part of $Z_{k, n}$, roughly speaking.  

For $k=2$ we will give an example of graph of $f_2$ in Figure \ref{fig2} below.

\begin{proof}
By \cite[Lemma 2.3 (2) and Lemma 3.3]{O1},  
there is a constant $c \le 1$ such that 
\begin{equation}\label{Lem1-mu}
\limsup_{n \to \infty} \frac{-\log_2 M_n(0, x)}{n} \le c \text{ $\mu_0$-a.s.$x$.} 
\end{equation}
If $\mu_0$ is singular, $c < 1$.  
If $\mu_0$ is absolutely continuous, $c=1$.

By (\ref{minmax}) and (\ref{Lem1-mu}) 
\begin{equation*}
\max\left\{  \frac{\left|f_k(x) - f_k(x_{n})\right|}{(x - x_{n})^{c}}, \frac{\left|f_k(x_{n}+2^{-n}) - f_k(x)\right|}{(x_{n}+2^{-n} -  x)^{c}} \right\} 
\ge \frac{1}{2} \left|Z_{k, n}(x)\right|, 
\end{equation*}  
for large $n$ and $\mu_{0}\text{-a.s.} x \in (0,1) \setminus D$.
(\ref{strict}) implies \[ \limsup_{n \to \infty}  |Z_{k, n}(x)| > 0 \] holds for any $x$.
Thus we have (i).   

Fix $l$ and $n$.
Denote $E^{\mu_0}_{[(l-1)/2^n, l/2^n)}$ by $E$. 
$Z_{k,m}^{+}$ and $Z_{k,m}^{-}$ denotes the positive and negative parts of $Z_{k, m}$. 
$Z_{k, m}^{+} - Z_{k,m}^{-} = Z_{k,m}$ and   $Z_{k, m}^{+} + Z_{k,m}^{-} = |Z_{k,m}|$.   
Using (\ref{Z-infty}) and that $\left\{\left|Z_{k, m}\right|\right\}_m$ is a submartingale, 
\[ \lim_{m \to \infty} E[Z_{k,m}^{+}] + E[Z_{k,m}^{-}] = \lim_{m \to \infty} E[|Z_{k,m}|] = \sup_{m} E[|Z_{k,m}|] = +\infty.\]     
Since $\{Z_{k, m}\}_{m \ge n}$ is a martingale, 
$E[Z_{k,m}^{+}] - E[Z_{k,m}^{-}] = E[Z_{k,n}]$ for any $m \ge n$. 
Therefore 
\begin{equation}\label{Div-Z} 
\lim_{m \to \infty} E[Z_{k,m}^{+}] = \lim_{m \to \infty} E[Z_{k,m}^{-}] =  +\infty. 
\end{equation}  

Let \[ A_{m} := \left\{x : H_m(0,x)  \le 2^{-1-c_{1} m} \right\}.\]   
By Azuma's inequality (\cite[Chapter E.14]{W})
there are constants $c_1 \in [0,1], c_2, c_3 \in (0, +\infty)$ 
such that 
for any $m$, $\mu_{0}( A_m ) \le c_2\exp(-c_3m)$.  
This and (\ref{M1-1}) imply
\[ E\left[Z_{k, m}^{+}, A_m\right] \le C_k m^{k} \frac{\mu_0(A_{m})}{\mu_0\left( [(l-1)/2^n, l/2^n) \right)} \to 0, \, \, m \to \infty. \] 
If $\left\{Z_{k, m}^{+} \ge E\left[Z_{k, m}^{+} \right]/2\right\} \subset A_{m}$ for large $m$, 
\[ \limsup_{m \to \infty} E\left[Z_{k, m}^{+}\right] \le 2 \limsup_{m \to \infty} E\left[Z_{k, m}^{+}, A_m\right] = 0.\]  
This contradicts  (\ref{Div-Z}).    
Hence 
\[ \left\{Z_{k, m}^{+} \ge \frac{E[Z_{k, m}^{+}]}{2} \right\} \cap A_{m}^{c} \cap \left[\frac{l-1}{2^n}, \frac{l}{2^n}\right) \ne \emptyset \] 
holds for infinitely many $m$.   
Using this, (\ref{minmax}) and  (\ref{Lem1-mu}),    
for $c$ in (\ref{Lem1-mu})  
\begin{equation}\label{MTNI-plus}
\sup_{x, y \in \left[(l-1)/2^n, l/2^n\right], x > y} \frac{f_k(x) - f_k(y)}{(x-y)^c} = +\infty.   
\end{equation}

Since $\lim_{m \to \infty} E[Z_{k,m}^{-}] =  +\infty$, 
there is  $c \in [0,1]$ such that  
\begin{equation}\label{MTNI-minus}
\inf_{x, y \in \left[(l-1)/2^n, l/2^n\right], x > y} \frac{f_k(x) - f_k(y)}{(x-y)^c} = -\infty.  
\end{equation} 
(\ref{MTNI-plus}) and (\ref{MTNI-minus}) imply (ii).    
\end{proof}

\begin{figure}[h]
\centering
\includegraphics[width = 60mm, height = 40mm, bb = 0 0 846 594]{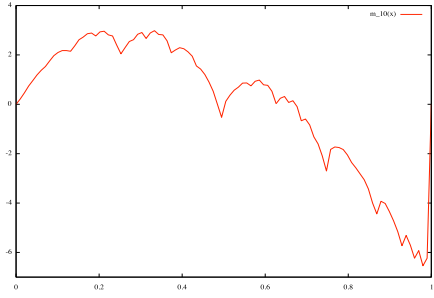}  
\caption{\small Graph of $f_2$ for $\displaystyle (b_1(t), c_0(t), c_1(t)) = \left(t + \frac{1}{3}, 0, 0\right)$.}  
\label{fig2}
\end{figure}


\section{Modulus of continuity} 

{\it In this section we always assume  (ND) holds and $k=1$. }  
First we will give some notation and lemmas. 
Second we will compare $\Delta_1 F(x,x+h)$ with $\log_2 (1/|h|)$ for $x \notin D$.  
Finally we will consider a modulus of continuity for $\Delta_1 F(x,x+h)$ at {\it$\mu_0$-a.s.$x$}.

Let \[ l(y, z) := \min\left\{i \ge 1 : X_{i}(y) \ne X_{i}(z) \right\}, \ \  y \ne z.\]    
Recall (\ref{Def-m}) for the definition of $m_1(z)$. 

\begin{Lem}[\text{\cite[Lemma 2]{Ko}}]\label{Kono-Lem}
Let $x \notin D$ and $h > 0$.   
Then \\
\textup{(i)} $\displaystyle \lim_{h \to 0, h > 0} l(x, x+h) = +\infty$. \\
\textup{(ii)} $l(x, x+h) \le m_1(h)$. \\
\textup{(iii)} $X_i(x) = X_i(x+h)$ for $1 \le i \le l(x,x+h)-1$. \\
\textup{(iv)} $X_{l(x, x+h)}(x) = 0$ and $X_{l(x, x+h)}(x+h) = 1$.\\
\textup{(v)} $X_i(x) = 1$ and $X_i(x+h) = 0$ for $l(x,x+h) < i \le m_1(h)-1$.  
\end{Lem}

Define 
\[ l_{x} := \min\left\{j > l(x, x+h) : X_j(x) = 0\right\} \text{ and } l_{x+h} := \min\left\{j > l(x,x+h) : X_{j}(x+h) = 1 \right\}.\]   
We have 
\begin{equation}\label{LLL}
(x+h)_{l_{x+h} - 1} = (x+h)_{l(x,x+h)} = x_{l_{x} - 1} + 2^{-(l_{x}-1)}.  
\end{equation}

\begin{Lem}[Key lemma]\label{Fund-LMC}
Let $x \notin D$ and $h > 0$. Then 
\begin{align}\label{Fund-f-Lem-1}
\Delta_1 F(x, x+h) &= \frac{F((x+h)_{l(x,x+h)}) - F(x)}{F(x+h)-F(x)} Z_{1,l_{x}}(x) \notag \\
&+ \frac{F(x+h) - F((x+h)_{l(x,x+h)})}{F(x+h)-F(x)} Z_{1, l_{x+h}}(x+h) + O(1).   
\end{align}
\end{Lem}

\begin{proof}
By (\ref{Lem1-2}) 
\[ f_1(x_{n+1}) - f_1(x_{n}) = X_{n+1}(x) \left( F(x_{n}+2^{-n-1}) - F(x_{n}) \right) \left(Z_{1,n}(x) + \overline{Y_n}(x) \right), \]   
\[ \text{where } \ \ \overline{Y_n}(x) := - \frac{H_n(0, x_n + 2^{-n-1})}{H_n(0, x_n)} Y_{n+1}(x_n + 2^{-n-1}).\]       

$\{\overline{Y_n}\}_n$ are bounded by (\ref{Basic1}) and (\ref{Fund-Y}).  
Summing up over $n$,  
\[ f_1(x) - f_1(x_{k}) = (F(x) - F(x_{k})) Z_{1,k}(x) + J(x, k) \]
\begin{equation*}
\text{where }  \ J(x, k) := \sum_{n = k}^{\infty} X_{n+1}(x) \left( F(x_{n}+2^{-n-1}) - F(x_{n}) \right) \left(\sum_{i=k+1}^{n} Y_i(x) + \overline{Y_n}(x) \right).  
\end{equation*} 
(\ref{Basic1}) and (\ref{Fund-Y}) imply $J(x, k) = O\left(F(x_{k}+2^{-k}) - F(x_{k})\right)$. 
Therefore 
\[ f_1(x+h) - f_1((x+h)_{l_{x+h} - 1}) = (F(x+h) - F((x+h)_{l_{x+h}-1})) \left(Z_{1, l_{x+h}}(x+h) + O(1)\right) \text{ and } \]
\[ f_1(x) - f_1(x_{l_{x}-1}) = (F(x) - F(x_{l_{x} - 1}))Z_{1,l_{x}-1}(x) + O\left(F(x_{l_x - 1}+2^{-(l_{x}-1)}) - F(x_{l_x - 1})\right).   \]
(\ref{LLL}) implies  
\[  f_1((x+h)_{l_{x+h} - 1}) -  f_1(x_{l_{x}-1}) = \left(F((x+h)_{l_{x+h}-1}) - F(x_{l_{x}-1})\right) Z_{1,l_{x}-1}(x). \]

Therefore 
\begin{align*}
f_1(x+h) -  f_1(x)  
&= \left(F((x+h)_{l_{x+h}-1}) - F(x)\right) Z_{1,l_{x}}(x) \notag\\
&+ \left(F(x+h) - F((x+h)_{l_{x+h}-1})\right) \left(Z_{1, l_{x+h}}(x+h) + O(1)\right) \notag\\
&+ O\left(F(x_{l_x - 1}+2^{-(l_{x}-1)}) - F(x_{l_x - 1})\right). 
\end{align*}   
Since $F(x_{l_x - 1}+2^{-(l_{x}-1)}) - F(x_{l_x - 1}) = O\left(F(x+h) - F(x)\right)$  
we have (\ref{Fund-f-Lem-1}).  
\end{proof}

\subsection{Modulus of continuity at non-dyadic rationals}  

Recall (\ref{Def-m}) for the definition of $\{m_n(z)\}_n$.   

\begin{Thm}\label{E-AK-JMAA} 
Assume $x \notin D$. 
Then 
\[ \lim_{h \to 0, h > 0} \frac{\Delta_1 F(x,x+h)}{\log_2 (1/h)} \text{ exists as a real number } \]
if and only if 
\[ \lim_{n \to \infty} \frac{m_{n+1}(1-x)}{m_{n}(1-x)} = 1 \text{ and } \lim_{n \to \infty} \frac{Z_{1,n}(x)}{n} \text{ exists.} \]
If they hold, 
\[ \lim_{h \to 0, h > 0} \frac{\Delta_1 F(x,x+h)}{\log_2 (1/h)} = \lim_{n \to \infty} \frac{Z_{1,n}(x)}{n}. \]
\end{Thm}

Considering also the limit from left 
we have Corollary \ref{bothside}, which extends \cite[Theorem 5.4]{AK-JMAA}.  
\cite{AK-JMAA} uses K\^ono's expression \cite[Lemma 3]{Ko}.  
If (L) holds then we may expect a counterpart of \cite[Lemma 3]{Ko}.    
However if (L) fails then it seems impossible to obtain a counterpart of \cite[Lemma 3]{Ko}.  
The key point is Lemma \ref{Fund-LMC} above,  
which states 
$\Delta_1 F(x, x+h)$ is between $Z_{1, l_{x}}(x)$ and $Z_{1, l_{x+h}}(x+h)$, roughly speaking.   
In Propositions \ref{GMC-1} and \ref{Large-fluc} below 
we will investigate the asymptotic of $\Delta_1 F(x, x+h) - Z_{1, l(x, x+h)}(x)$.      
These results are different depending on the asymptotic for $\dfrac{m_{n+1}(1-x)}{m_{n}(1-x)}$.   
Eliminate parts consisting of the differentials by estimating $Z_{1, l_{x}}(x) - Z_{1, l(x, x+h)}(x)$ and $Z_{1, l_{x+h}}(x+h) - Z_{1, l(x, x+h)}(x)$. 
Finally consider quantities expressed by $F$, $x$ and  $h$ only as in (\ref{easy}) and (\ref{difficult}) below.

\begin{Prop}\label{GMC-1}
If $x \notin D$ and $\displaystyle\lim_{n \to \infty} \dfrac{m_{n+1}(1-x)}{m_n(1-x)} = 1$,
\[ \lim_{h \to 0, h > 0} \frac{\Delta_1 F(x,x+h) - Z_{1, \lfloor \log_2 (1/h) \rfloor}(x)}{\log_2 (1/h)} = 0. \]
Here $\lfloor x \rfloor$ denotes the maximal integer less than or equal to $x$. 
\end{Prop}

\begin{proof}
First we remark $\displaystyle \lim_{h \to 0, h > 0} \frac{l(x,x+h)}{m_1(h)} = 1$ by the assumptions.  
By (\ref{Fund-Y}), (\ref{ZY}) and  (\ref{Fund-f-Lem-1}) 
it suffices to show 
\begin{equation}\label{easy}
\lim_{h \to 0, h > 0} \frac{F((x+h)_{l_{x+h}-1}) - F(x)}{F(x+h)-F(x)} \cdot \frac{l_{x} - m_1(h)}{m_1(h)} = 0 \text{ and }
\end{equation}
\begin{equation}\label{difficult}
 \lim_{h \to 0, h > 0} \frac{F(x+h) - F((x+h)_{l_{x+h}-1})}{F(x+h)-F(x)} \cdot \frac{ l_{x+h} - l(x,x+h) + m_1(h) - l(x,x+h)}{m_1(h)} = 0. 
\end{equation}

Using Lemma \ref{Kono-Lem} and $\displaystyle\lim_{n \to \infty} \dfrac{m_{n+1}(1-x)}{m_n(1-x)} = 1$  
we have (\ref{easy}).     

We now show (\ref{difficult}). 
We will give an upper bound for 
\[ (*) := \frac{F(x+h) - F((x+h)_{l_{x+h}-1})}{F(x+h)-F(x)} \times \left(\frac{l_{x+h}}{l(x,x+h)} - 1\right).  \]
If $l_{x+h}  \le (1+\epsilon)l(x,x+h)$,  $(*) \le \epsilon$. 

Using (\ref{Basic2}) and \[ M_{l(x,x+h)-1}\left(0, (x+h)_{l_{x+h}-1}\right) =  M_{l(x,x+h)-1}\left(0, x_{l_{x}} + 2^{-l_{x}}\right), \]     
there are constants $0 < c < 1 \le C < +\infty$ such that 
\begin{align*} 
\frac{F(x+h) - F((x+h)_{l_{x+h}-1})}{F(x+h)-F(x)} &\le \frac{M_{l_{x+h}-1}\left(0, (x+h)_{l_{x+h}-1}\right)}{M_{l_{x}}\left(0, x_{l_{x}} + 2^{-l_{x}}\right)} \\
&= C^{l_x - l(x,x+h)} c^{l_{x+h} - l(x,x+h)} 
\end{align*}    

For any $\epsilon > 0$  
there is $\delta(\epsilon) > 0$ with $c^{\epsilon/2} < C^{\delta(\epsilon)}$.    
Therefore if $l_{x+h}  \ge (1+\epsilon)l(x,x+h)$ and $h$ is sufficiently small,  $l_{x} \le (1+\delta(\epsilon))l(x,x+h)$ and  
\[  \left(\frac{l_{x+h}}{l(x,x+h)} - 1\right) \left(\frac{c^{l_{x+h}/l(x,x+h) - 1}}{C^{l_{x}/l(x,x+h) - 1}} \right)^{l(x,x+h)} 
\le \frac{(l_{x+h}  - l(x,x+h)) c^{(l_{x+h} - l(x,x+h))/2}  }{l(x,x+h)}. \]
Hence $(*) \le \epsilon$. 
Thus we have (\ref{difficult}).  
\end{proof}

\begin{Prop}\label{Large-fluc}
If $x \notin D$ and $\displaystyle \limsup_{n \to \infty} \dfrac{m_{n+1}(1-x)}{m_{n}(1-x)} > 1$, 
\[ \frac{\Delta_1 F(x, x+h)}{\log_2 (1/h)} \text{ fluctuates if $h \to 0, h > 0$.  }\] 
\end{Prop}

\begin{proof}
Let $\delta > 0$ and $(n(j))_j$ be an increasing sequence satisfying \[ m_{n(j)}(1-x) \ge m_{n(j) - 1}(1-x) (1+\delta).\]     

Assume $b_1(1+c_0)(1+c_1) \ge 1-b_1$. 
Let \[ m(1,j) := m_{n(j)}(1-x) - 2 \text{ and } m(2,j) := m_{n(j)}(1-x).\]   
Then 
\begin{equation}\label{FM1} 
F(x+2^{-m(1,j)}) - F(x) \le M_{m(1,j)}(0, x_{m(1,j)}) + M_{m(1,j)}\left(0, x_{m(1,j)}+2^{-m(1,j)}\right), 
\end{equation}
\begin{equation}\label{FM2} 
F(x+2^{-m(1,j)}) - F\left((x+2^{-m(1,j)})_{m(1,j)}\right) \ge M_{m(1,j)}\left(0, x_{m(1,j)}+2^{-m(1,j)}\right) \text{ and } 
\end{equation}      
\begin{equation}\label{MM1} 
\frac{M_{m(1,j)}\left(0, x_{m(1,j)}+2^{-m(1,j)}\right)}{M_{m(1,j)}(0, x_{m(1,j)})} 
= \prod_{i=m_{n(j) - 1}(1-x)}^{m_{n(j)}(1-x) - 1} \frac{H_{i}\left(0, x_{m(1,j)}+2^{-m(1,j)}\right)}{H_{i}(0, x_{m(1,j)})}.  
\end{equation} 

By Lemma \ref{H-exp-fast-right} (i) and $m_{n(j)}(1-x) - m_{n(j) - 1}(1-x) \ge (1+\delta) j$,  
\begin{equation}\label{H-conv-1} 
\lim_{j \to \infty} \prod_{i= m_{n(j) - 1}(1-x)}^{m_{n(j)}(1-x) - 1} \frac{H_{i}\left(0, x_{m(1,j)}+2^{-m(1,j)}\right)}{b_1(c_0 + 1)} = 1.   
\end{equation}  
\begin{equation}\label{H-conv-2}
\lim_{j \to \infty} \prod_{i= m_{n(j) - 1}(1-x)}^{m_{n(j)}(1-x) - 1} \frac{H_{i}(0, x_{m(1,j)})(1+c_0)}{1-b_1} = 1.   
\end{equation}  

Using (\ref{FM1}), (\ref{FM2}), (\ref{MM1}), (\ref{H-conv-1}), (\ref{H-conv-2}) and  $b_1(1+c_0)(1+c_1) \ge 1-b_1$, 
\begin{equation}\label{F-lower}
\liminf_{j \to \infty} \frac{F(x+2^{-m(1,j)}) - F\left((x+2^{-m(1,j)})_{m(1,j)}\right)}{F(x+2^{-m(1,j)}) - F(x)} \ge \frac{1}{2}. 
\end{equation} 

Since \[ m(1,j) - l(x, x+2^{-m(1,j)}) = m_{n(j)}(1-x) - m_{n(j)-1}(1-x) - 2,\]   
\[ \liminf_{j \to \infty} \frac{m(1,j) - l(x, x+2^{-m(1,j)})}{m(1,j)} > 0. \] 

Using this, (\ref{F-lower}) and  (\ref{Fund-f-Lem-1}), 
\begin{equation}\label{m-1-j} 
\liminf_{j \to \infty} \frac{\Delta_1 F(x, x+2^{-m(1,j)}) - Z_{1, m(1,j)}(x)}{m(1,j)} > 0. 
\end{equation}

Recall $m(2,j) = l(x, x+2^{-m(2,j)})$ and (\ref{Fund-f-Lem-1}).  
Considering the cases $X_{m(2,j)}(x) = 0$ and $X_{m(2,j)}(x) = 1$ respectively, 
\begin{equation}\label{m-2-j} 
\limsup_{j \to \infty} \frac{\Delta_1 F(x, x + 2^{-m(2,j)})  - Z_{1, m(2,j)}(x)}{m(2,j)} \le 0. 
\end{equation}

By  (\ref{Fund-Y}), 
\begin{equation}\label{Z-small} 
\lim_{j \to \infty} \frac{Z_{1, m(1,j)}(x)}{m(1,j)} -  \frac{Z_{1, m(2,j)}(x)}{m(2,j)} = 0.  
\end{equation} 
Using (\ref{m-1-j}), (\ref{m-2-j}) and  (\ref{Z-small}), we have the assertion.

If $b_1(1+c_0)(1+c_1) < 1-b_1$, by Lemma \ref{H-exp-fast} (ii)    
there are $c^{\prime}, c^{\prime\prime}  > 0$ and $\delta_1, \delta_2 \in (0, \delta)$ such that 
for large $j$ 
\[ \Delta_1 F\left(x, x+2^{-(1+\delta_1) n(j-1)}\right) \ge Z_{1, n(j-1)}(x) + c^{\prime} \delta_1 n(j-1) \text{ and } \]
\[ \Delta_1 F\left(x, x+2^{-(1+\delta_2) n(j-1)}\right) \le Z_{1, n(j-1)}(x) - c^{\prime\prime}  \delta_2 n(j-1).  \]
For large $j$, 
\[ \frac{\Delta_1 F\left(x, x+2^{-(1+\delta_1) n(j-1)}\right)}{(1+\delta_1) n(j-1)} - \frac{\Delta_1 F(x, x+2^{-n(j-1)})}{n(j-1)}  
\ge \frac{c^{\prime}}{2} \delta_1 \, \text{ if $Z_{1, n(j-1)}(x) < 0$.} \]
\[ \frac{\Delta_1 F\left(x, x+2^{-(1+\delta_2) n(j-1)}\right)}{(1+\delta_2) n(j-1)} - \frac{\Delta_1 F(x, x+2^{-n(j-1)})}{n(j-1)}  
\le -\frac{c^{\prime\prime} }{2} \delta_2  \, \text{ if $Z_{1, n(j-1)}(x) > 0$. } \]
Thus we have the assertion. 
\end{proof}

Theorem \ref{E-AK-JMAA} follows from Propositions \ref{GMC-1} and \ref{Large-fluc}.    

Note \[ \Delta_1 F(x, x+h) =  \Delta_1 \widetilde F(1-x, 1-x-h) \text{ and } \widetilde Z_{1, n}(1-x) = Z_{1, n}(x),  \ \  x \notin D. \]  
We can consider $\displaystyle\lim_{h \to 0, h < 0} \dfrac{\Delta_1 F(x, x+h)}{\log_2 (1/|h|)}$ in the same manner. We have   
\begin{Cor}\label{bothside} 
Assume  $x \notin D$.   
Then  
\[ \lim_{h \to 0} \frac{\Delta_1 F(x, x+h)}{\log_2 (1/|h|)} \text{ exists } \] 
if and only if 
\[ \lim_{n \to \infty} \frac{m_{n+1}(1-x)}{m_{n}(1-x)} =  
\lim_{n \to \infty} \frac{m_{n+1}(x)}{m_{n}(x)} = 1 \text{ and }
\lim_{n \to \infty} \frac{Z_{1,n}(x)}{n} \text{ exists.} \] 
\end{Cor}

\subsection{Modulus of continuity at $\mu_0$-a.s. points}

\begin{Thm}\label{LMC} 
There are two constants $0 < c \le C < +\infty$ such that the following hold for $\mu_0$-a.s. $x$ :   
\begin{equation}\label{sup} 
c \le \limsup_{h \to 0} \frac{\Delta_1 F(x, x+h)}{(\log_2 (1/|h|) \log \log \log_2 (1/|h|))^{1/2}} \le C. 
\end{equation} 
\begin{equation}\label{inf} 
-C \le \liminf_{h \to 0} \frac{ \Delta_1 F(x, x+h)}{(\log_2 (1/|h|) \log \log \log_2 (1/|h|))^{1/2}} \le -c. 
\end{equation} 
\end{Thm}

By this we can improve (\ref{MTNI-k})  for $k=1$ as follows\footnote{\text{  If } \begin{equation*}\label{strong}
\limsup_{n \to \infty} Z_{k, n}^{+}(x) > 0 \text{ and } \limsup_{n \to \infty} Z_{k, n}^{-}(x) > 0 \ \text{ $\mu_0$-a.s.$x$,} \tag{**}
\end{equation*} 
\[ \limsup_{h \to 0} \frac{f_k(x+h) - f_k(x)}{|h|^c}  = +\infty = - \liminf_{h \to 0} \frac{f_k(x+h) - f_k(x)}{|h|^c} \text{ $\mu_0$-a.s.$x$.}\]
If we can apply Lemma \ref{LIL} to $\{Z_{k, n}\}_{n}$ for $k \ge 2$,  (\ref{strong}) follows immediately and moreover Theorem \ref{LMC} holds for $k \ge 2$. } :    
\[ \limsup_{h \to 0} \frac{f_1(x+h) - f_1(x)}{|h|^c} = +\infty = -\liminf_{h \to 0} \frac{f_1(x+h) - f_1(x)}{|h|^c}  \text{  $\mu_0$-a.s.$x$.} \]
As we will see in Corollaries \ref{LMC-lower-AC} and \ref{Ori-T}, 
for some special choices of $(b_1(t), c_0(t), c_1(t))$, 
we can improve (\ref{sup}) and (\ref{inf}).    

The key tools of the lower bound for (\ref{sup}) and the upper bound for (\ref{inf}) are  (\ref{minmax}) and Stout's law of the iterated logarithm (LIL) below.  
Recall (\ref{ZY}). 
Apply Stout's law of the iterated logarithm (LIL) below to $\{Y_i\}_i$.  
The key tool of the upper bound for (\ref{sup}) and the lower bound for (\ref{inf}) 
is Lemma \ref{Fund-LMC} above, which states $\Delta_1 F(x, x+h)$ is between $Z_{1, l_{x}}(x)$ and $Z_{1, l_{x+h}}(x+h)$, roughly.  
Estimate parts consisting of the differentials by estimating $Z_{1, l_{x}}(x) - Z_{1, l(x, x+h)}(x)$ and $Z_{1, l_{x+h}}(x+h) - Z_{1, l(x,x+h)}(x)$. 
Now apply Stout's LIL to $Z_{1, l(x,x+h)}(x)$ and these differences.   

\begin{Lem}[Stout's LIL for martingales \cite{St}]\label{LIL} 
Let $(\Omega, \mathcal{F}, P)$ be a probability space and $\{S_n, \mathcal{F}_n\}_{n \ge 0}$ be a martingale on it.  
Let $I_n := \sum_{i=1}^{n} E[(S_i - S_{i-1})^2| \mathcal{F}_{i-1}]$ where we denote the expectation with respect to $P$ by $E$. 
Assume  there are constants $0 < c \le C < +\infty$ such that $c \le |S_i - S_{i-1}| \le C$  $\mu_0$-a.s. for any $i \ge 1$. 
Then 
\begin{equation*}
\limsup_{n \to \infty} \frac{S_{n}}{(I_n \log\log I_n )^{1/2}} = \sqrt{2} = -\liminf_{n \to \infty} \frac{S_n}{(I_n \log\log I_n )^{1/2}} \ \text{ $P$-a.s.}      
\end{equation*}
\end{Lem}

\begin{proof}[Proof of Theorem \ref{LMC}]
First we show the lower bound for (\ref{sup}) and the upper bound for (\ref{inf}).    
Recall (\ref{minmax}).   
Applying Lemma \ref{LIL} to $\{Y_i\}_i$, 
there is $c > 0$ such that the following hold $\mu_0$-a.s.$x$ : 
\begin{equation}\label{LIL-sup}   
\limsup_{n \to +\infty} \frac{\max\left\{\Delta_1 F(x, x_n), \Delta_1 F(x, x_n + 2^{-n})\right\}}{(n \log\log n)^{1/2}} \ge c. 
\end{equation}  
\begin{equation}\label{LIL-inf} 
\liminf_{n \to +\infty} \frac{\min\left\{\Delta_1 F(x, x_n), \Delta_1 F(x, x_n + 2^{-n})\right\}}{(n \log\log n)^{1/2}} \le -c.   
\end{equation}   

By \cite[Lemma 3.2]{O1}   
there are constants $0 < c^{\prime}  \le c^{\prime\prime} < 1$ such that 
\[ c^{\prime} \le \liminf_{n \to +\infty} \frac{\left|\{i \in \{1, \dots, n\}: X_i(x) = 0 \}\right|}{n} \le \limsup_{n \to +\infty} \frac{\left|\{i \in \{1, \dots, n\}: X_i(x) = 0 \}\right|}{n}  \le c^{\prime\prime} \]
holds $\mu_0$-a.s. $x$.  

Let $\sigma(h) := \log_2 \left(1/|h|\right)$ for $h \ne 0$.  
Using this and (\ref{Fund-Y}), 
there is  $C < +\infty$ such that  
\begin{equation}\label{LIL-Bdd}  
\limsup_{n \to \infty} \frac{\sigma(x-x_n) + \sigma(x_n + 2^{-n} - x)}{\sigma(2^{-n})} \le C   \text{ $\mu_0$-a.s.$x$.}    
\end{equation} 

(\ref{LIL-sup}) and (\ref{LIL-Bdd}) imply the lower bound for (\ref{sup}). 
(\ref{LIL-inf}) and (\ref{LIL-Bdd}) imply the upper bound for (\ref{inf}).   

Second we show the upper bound for (\ref{sup}) and the lower bound for (\ref{inf}). 
Assume $h > 0$. 
Applying Lemma \ref{LIL} to $\left\{\sum_{i=1}^{n} X_i - E^{\mu_0}[X_i | \mathcal{F}_{i-1}]\right\}_n$ and $\{Y_i\}_i$,     
\[ l_{x} - l(x, x+h) + l_{x+h} - l(x, x+h) = O\left(( l(x, x+h) \log\log  l(x, x+h))^{1/2}\right) \text{ and }  \]
\[ Z_{1, l(x,x+h)}(x) = O\left((l(x,x+h) \log\log l(x,x+h))^{1/2} \right),   h \to 0, h > 0,  \text{ $\mu_0$-a.s.$x$.} \] 
Using (\ref{Fund-f-Lem-1})       
\begin{equation}\label{LMC-right-2} 
\Delta_1 F(x, x+h) = O\left((\sigma(h) \log\log \sigma(h))^{1/2}\right), \, h \to 0, h > 0,  \text{ $\mu_0$-a.s.$x$.}   
\end{equation} 

Now assume  $h < 0$.       
Applying Lemma \ref{LIL} 
to $\{\widetilde Y_i\}_i$,      
\[ \Delta_1 \widetilde F(y, y+h) = O\left((\sigma(h) \log\log \sigma(h))^{1/2}\right), \, h \to 0, h > 0,  \text{  $\widetilde\mu_0$-a.s.$y$.}   \]
Let $T(y) := 1-y$. 
By (\ref{Dual}), $\widetilde \mu_0 = \mu_0 \circ T^{-1} \text{ and }$
\begin{equation}\label{Lem2-3} 
  \Delta_k \widetilde F(x, y) = \Delta_k F(1-x, 1-y),  \ \ x, y \in (0,1), \, k \ge 1.  
\end{equation}

Therefore 
\begin{equation}\label{LMC-left-2}
\Delta_1 F(x, x+h) 
=  \Delta_1 \widetilde F(1-x, 1-x-h) = O\left((\sigma(h) \log\log \sigma(h))^{1/2}\right), \, h \to 0, h < 0,  \text{ $\mu_0$-a.s.$x$.}  
\end{equation}  
(\ref{LMC-right-2}) and (\ref{LMC-left-2}) complete the proof of the upper bound for (\ref{sup}) and the lower bound for (\ref{inf}).       
\end{proof}

Let \[ I_{n}(x) := \sum_{i=1}^{n} E\left[Y_i^2 | \mathcal{F}_{i-1}\right](x) \text{ and } \sigma(h,x) := I_{\lfloor \log_2 (1/|h|)\rfloor}(x).\]  

\begin{Cor}\label{LMC-lower-AC}  
If $\mu_0$ is absolutely continuous,   
\begin{equation}\label{LMC-AC-inf} 
\liminf_{h \to 0} \frac{\Delta_1 F(x,x+h)}{\left(\sigma(h, x) \log\log \sigma(h,x)\right)^{1/2}} 
= -\sqrt{2} \ \ \  \text{ $\mu_0$-a.s. $x$.}  
\end{equation}  
\end{Cor}

\begin{proof}
Assume  $x \notin D$ and $h > 0$.
Using Lemma \ref{Fund-LMC} and that \[ Z_{1, l_{x+h}}(x+h) \ge Z_{1, l(x, x+h)}(x+h) =  Z_{1, l(x, x+h)}(x) + O(1),\]     
\[ \Delta_1 F(x,x+h) \ge \min\left\{ Z_{1, l_x}(x),  Z_{1, l(x,x+h)}(x) \right\} + O(1). \]

Since $\displaystyle\lim_{h \to 0, h > 0} \dfrac{l_{x}}{l(x,x+h)} = 1$,   
\[ \lim_{h \to 0, h > 0} \frac{I_{l(x,x+h)}(x)}{\sigma(h, x)} =  \lim_{h \to 0, h > 0} \frac{I_{l_{x}}(x)}{\sigma(h, x)}  = 1  \ \ \  \text{ $\mu_0$-a.s.$x$. }  \]
Therefore 
\[ \liminf_{h \to 0, h > 0} \frac{\Delta_1 F(x,x+h)}{\left(\sigma(h, x) \log\log \sigma(h,x)\right)^{1/2}} 
\ge -\sqrt{2} \ \ \ \text{ $\mu_0$-a.s.$x$.}  \]  
By (\ref{Lem2-3}) 
\[ \liminf_{h \to 0, h < 0} \frac{\Delta_1 F(x,x+h)}{\left(\sigma(h, x) \log\log \sigma(h,x)\right)^{1/2}} 
\ge -\sqrt{2} \ \ \ \text{ $\mu_0$-a.s.$x$}  \]  
in the same manner.  
Thus we have  
\begin{equation*}
\liminf_{h \to 0} \frac{\Delta_1 F(x,x+h)}{\left(\sigma(h, x) \log\log \sigma(h,x)\right)^{1/2}} 
\ge -\sqrt{2} \ \ \ \text{ $\mu_0$-a.s.$x$.}  
\end{equation*}  

If $\mu_0$ is absolutely continuous, 
\[ \lim_{n \to \infty} \frac{\sigma(x-x_n, x)}{I_n(x)} = \lim_{n \to \infty} \frac{\sigma(x_n + 2^{-n} - x, x)}{I_{n}(x)} = 1 \ \ \  \text{ $\mu_0$-a.s.$x$.} \]  
Using this, Lemma \ref{LIL} and (\ref{minmax}),  
we have the upper bound of (\ref{LMC-AC-inf}).  
\end{proof}

If $b_1(t) = t+ \dfrac{1}{2}$ and $c_0(t) = c_1(t) = 0$, by symmetry, 
\[ \partial_t F(0, x) = \partial_t F(0, 1-x), F(x) = 1- F(1-x) \text{ and }  \sigma(h, x) = \lfloor \log_2 (1/|h|) \rfloor.\]        
Hence 
\begin{Cor}[The original Takagi function case of \text{\cite[Theorem 5]{Ko}}]\label{Ori-T}
If $b_1(t) = t+ \dfrac{1}{2}$ and 
$c_0(t) = c_1(t) = 0$, the following hold $\mu_0$-a.s. $x$ : 
\[ \limsup_{h \to 0} \frac{\Delta_1 F(x,x+h)}{\left(\log_2 (1/|h|) \log\log \log_2 (1/|h|) \right)^{1/2}} 
= \sqrt{2} = - \liminf_{h \to 0} \frac{\Delta_1 F(x,x+h)}{\left(\log_2 (1/|h|) \log\log \log_2 (1/|h|) \right)^{1/2}}. \]
\end{Cor}

\section*{Acknowledgement}   
The author would like to express his thanks to the referee for careful reading of the manuscript and many valuable suggestions. 
He also would like to express his thanks to E. de Amo for comments on an earlier version of this paper. 
This work was partly supported by Grant-in-Aid for JSPS fellows (24.8491).


\end{document}